\documentclass[12pt]{amsart}
\usepackage[left=1in,top=1.2in,bottom=1in]{geometry}
\usepackage{tikz}
\usepackage{url} 
\usepackage{color}
\usepackage[pdfencoding=auto, psdextra]{hyperref}
\usepackage{cleveref}
\Crefname{figure}{Figure}{Figure}

\newcommand{\Hthree}{\mathbb{H}^3}
\newcommand{\Pthree}{\mathcal{P}}
\newcommand{\Ttwo}{\mathcal{T}}

\newcommand{\ZZ}{\mathbb{Z}}
\newcommand{\ZZtwelve}{\ZZ[\zeta_{12}]}

\newcommand{\mI}{m_\infty}
\newcommand{\mZero}{m_0}

\theoremstyle{plain}
\newtheorem{counter}{}[section]
\newtheorem{theorem}[counter]{Theorem}

\newtheorem{fact}[counter]{Fact}
\newtheorem{prop}[counter]{Proposition}
\newtheorem{lemma}[counter]{Lemma}
\newtheorem{corollary}[counter]{Corollary} 
\newtheorem{question}[counter]{Question} 

\theoremstyle{definition}
\newtheorem{definition}[counter]{Definition} 
\newtheorem{example}[counter]{Example} 
\newtheorem{claim}{Claim}[counter]
\newtheorem{remark}[counter]{Remark}
\newtheorem*{notation}{Notation}

\newcommand\aceofbase{\ensuremath{\mathcal{R}}}
\newcommand\totaleclipse{\ensuremath{\mathcal{S}}}

\title{Mixed-platonic $3$-manifolds}
\author[E. Chesebro]{Eric Chesebro}
\address{Department of Mathematical Sciences\\
University of Montana\\
32 Campus Drive \\
Missoula, MT 59812}
\email{eric.chesebro@umt.edu}

\author[M. Chu]{Michelle Chu}
\address{School of Mathematics\\ University of Minnesota \\ 
127 Vincent Hall 206 Church St. SE\\ Minneapolis, MN 55455}
\email{mchu@umn.edu}

\author[J. DeBlois]{Jason DeBlois}
\address{Department of Mathematics\\
University of Pittsburgh\\
301 Thackeray Hall\\
Pittsburgh, PA 15260}
\email{jdeblois@pitt.edu}

\author[N. Hoffman]{Neil R Hoffman}
\address{Department of Mathematics and Statistics \\
University of Minnesota Duluth\\
140 Engineering Building\\
1303 Ordean Ct\\
Duluth, MN 55812}
\email{neilhoff@d.umn.edu}

\author[P. Mondal]{Priyadip Mondal}
\address{Department of Mathematics\\
Rutgers University-New Brunswick\\
Hill Center for the Mathematical Sciences\\ 
110 Frelinghuysen Road\\
Piscataway, NJ 08854 }
\email{pm868@scarletmail.rutgers.edu}

\author[G. Walsh]{Genevieve S Walsh}
\address{Department of Math\\Tufts University\\177 College Ave\\
Medford, MA 02155}
\email{genevieve.walsh@tufts.edu}

\begin{document}

\begin{abstract}
    We introduce a class of cusped hyperbolic $3$-manifolds that we call mixed-platonic, composed of regular ideal hyperbolic polyhedra of more than one type, which includes certain previously-known examples. We establish basic facts about mixed-platonic manifolds which allow us to conclude, among other things, that there is no mixed-platonic hyperbolic knot complement with hidden symmetries.
\end{abstract}

\maketitle

\section{Introduction}

Three classical and beautiful hyperbolic knot complements in $S^3$ 
are \textit{platonic}, meaning they admit a complete hyperbolic structure that decomposes as a union of isometric copies of a single regular ideal polyhedron.  Thurston famously showed that the figure-eight knot complement decomposes into two regular ideal tetrahedra \cite[Ch.~3]{Th_notes}, cf.~\cite{Thurstonbook}; and Aitchison-Rubinstein discovered the two ``dodecahedral knots'', each of whose complement decomposes into two regular ideal dodecahedra \cite{AitchRubdodec}. There are many more platonic hyperbolic manifolds---a census is given in \cite{Matthiascensus}---including the complements of the Whitehead link and Borromean rings which were also described in \cite[Ch.~3]{Th_notes}. However, Hoffman showed that the complements of the figure-eight and the two known dodecahedral knots are the only platonic hyperbolic knot complements in $S^3$ \cite[Corollary 2]{NeilExperiment}, building on work of Reid who proved that the only arithmetic hyperbolic knot complement is that of the figure-eight \cite{reidarith}.

The three platonic knot complements are also known to have {\it hidden symmetries}, meaning that each has an isometry between finite-degree covers that does not descend to the knot complement itself.  These are the only hyperbolic knot complements in $S^3$ known to have hidden symmetries, and in 1995, Neumann-Reid conjectured that they are the only three \cite[Problem 3.64(a)]{KirbyList}. The conjecture remains open, although many knot complements have been shown not to have hidden symmetries.   

Since platonic knots are exactly the knots which are known to admit hidden symmetries, it is natural to investigate the knot complements which are made of at least two types of platonic solids. In this paper we name and initiate a systematic study of \textit{mixed-platonic} manifolds, complete hyperbolic $3$-manifolds that decompose into regular ideal polyhedra of more than one type. Section \ref{tinderbox} below gives the formal definition and exhibits some examples, including the knot complement $S^3 - 12n706$, that have been known to experts for some time. A certain structural feature of the cusp of $S^3-12n706$ might suggest that mixed-platonic knot complements are a promising class to search for more examples with hidden symmetries. However, the main result of this paper shows that this is not so.

\newcommand\NoMPKHidSym{No mixed-platonic knot complement in $S^3$ has hidden symmetries.}
\theoremstyle{plain}
\newtheorem*{NoMPKHS}{\Cref{main_thm}}
\begin{NoMPKHS}\NoMPKHidSym\end{NoMPKHS}

The proof of this result, developed over the course of Sections \ref{basic} through \ref{eclipse}, entails a rich blend of arithmetic, combinatorics, and geometry. Section \ref{basic} establishes basic results about mixed-platonic orbifolds, notably the existence of a standard horoball packing (\Cref{packing}) and tight control over the translation length of peripheral elements representing meridians of a mixed-platonic knot complement (\Cref{neilolarry}, proved number-theoretically). Section \ref{perti} focuses on \textit{peripheral tilings} of horoball boundaries induced by the polyhedra comprising the decomposition of a mixed-platonic manifold. Every peripheral tiling is a tiling of the plane by equilateral triangles and squares having lattice symmetry. Section \ref{perti} gives criteria showing that not every such tiling can be the peripheral tiling of a mixed-platonic orbifold.

The main result of Section \ref{Pgh Regional Transit}, \Cref{And then there were two} classifies the tilings of the plane by equilateral triangles and squares with sidelength $1$, having a short translational symmetry and also order-three rotational symmetry. Peripheral tilings of mixed-platonic knot complements with hidden symmetries would have these properties, as we now discuss. The translation length bound follows from \Cref{neilolarry} here, and the existence of rotations follows from previous work: by \cite[Prop.~9.1]{NeumannReidArith}, every hyperbolic knot complement $M$ with hidden symmetries covers a rigid-cusped orbifold $O$; if $M$ is mixed-platonic then by \Cref{commensurator} here, $O$ is as well, with the same peripheral tiling as $M$. Theorem 1.1 of \cite{NeilOrbiCusps} further implies in this case that the orbifold fundamental group of the cusp of $O$ contains order-three rotations. 

\Cref{And then there were two} describes three tilings, two of which were showed in Section \ref{perti} not to be the peripheral tiling of any one-cusped mixed-platonic manifold. The main result of Section \ref{eclipse}, \Cref{slop} asserts that the full symmetry group of the third tiling---a $(2,3,6)$-rotation group---is not the peripheral subgroup of a mixed-platonic orbifold, and neither is its index-two $(3,3,3)$-rotation subgroup. The proof of \Cref{main_thm}, which uses the prior results listed above, concludes this section.

While we have showed that the collection of mixed-platonic manifolds does not contain a knot complement with hidden symmetries, our work exhibits numerous nice properties of its members, and it has prompted many further questions.  \Cref{sec:questions} lists some of these.

\section*{Acknowledgements}

This project originated at a SQuaRE which concluded in March 2024 at the American Institute of Mathematics (AIM) in Pasadena, CA.  
The authors thank AIM for providing a supportive and mathematically rich environment. 

The second author was partially supported by NSF grant DMS-2243188 and the sixth author by NSF grant DMS-2005353.

\section{Definition and existing examples}\label{tinderbox}

\begin{definition}\label{mtothep} A connected, complete, finite-volume hyperbolic $3$-orbifold $M=\Hthree/\Gamma$ is \textit{mixed-platonic} if there is a $\Gamma$-invariant tiling of $\Hthree$ with $n_1$ $\Gamma$-orbits of regular ideal tetrahedra, and $n_2$ of regular ideal octahedra, such that $n_1 \geq 1$ and $n_2 \geq 1$. 
For convenience, the tetrahedra and octahedra in a $\Gamma$-invariant decomposition as above are called \textit{tiles}. \end{definition}

This definition is more restrictive than it absolutely has to be, as we now briefly address.

\begin{remark} We have restricted to finite-volume orbifolds out of convenience (specifically, for counting tiles), since the examples of interest in this paper have finite volume.
\end{remark}

\begin{remark} It is natural to contemplate an \textit{a priori} broader definition than \Cref{mtothep}, in which regular ideal cubes, dodecahedra, or icosahedra could also occur as tiles. However in such a tiling, a cube (or, respectively, a dodecahedron) could only abut another cube (resp.~dodecahedron) along a face, since it is the only platonic solid with quadrilateral (resp.~pentagonal) faces. A regular ideal icosahedron does have ideal triangular faces, but 
it has a dihedral angle of $3\pi/5$ 
around every edge. For an edge of a putative tiling of $\mathbb{H}^3$ by regular ideal tetrahedra, octahedra, and/or icosahedra, if $T,O,I$ respectively count the number of tetrahedra, octahedra, and icosahedra incident to this edge we must have $\pi/3\, T + \pi/2\, O + 3\pi/5\, I = 2\pi$. But all non-negative integer solutions to this have $I=0$. Thus like cubes and dodecahedra, icosahedra do not appear as tiles of mixed platonic manifolds.
\end{remark}

\begin{definition} We call a knot or link in $S^3$ \textit{mixed-platonic} if its complement admits the structure of a mixed-platonic complete hyperbolic $3$-manifold.\end{definition} 

Given the constraint $n_1 \geq 1$ and $n_2 \geq 1$, we have that the figure-eight knot complement and the Whitehead link complement are not mixed-platonic. However, there are mixed-platonic manifolds that have been known for some time: 

\begin{example}\label{ex:mph}
    There are examples of mixed-platonic manifolds described in \cite[Section 5.6]{book} that are hyperbolic link complements.  These are constructed as ``hybrids'' of the Whitehead link (tiled by ideal  octahedra) and a 4-component chain link (tiled by ideal tetrahedra): by cutting along embedded 3-punctured spheres in each (which are necessarily totally geodesic) and gluing the resulting pieces along those 3-punctured spheres. In each original link complement, the three-punctured sphere to be cut along is a union of faces of the link complement's  decomposition into regular polyhedra; thus the link complement produced by gluing inherits its polyhedral decomposition from those of the originals.
\end{example}

\begin{example}\label{ex: Boyd knot}
    The complement of the knot pictured on the left in \Cref{fig:Boyd} is a mixed-platonic manifold. {This knot has the name $12n706$ in the census ``\texttt{HTLinkExteriors}'' that ships with SnapPy \cite{SnapPy}}; 
    in particular, it has $12$ crossings and is non-alternating. It was studied by D. W. Boyd and also appears in \cite[\S7.1]{GoodmanHeardHodgson}. This knot complement decomposes into 6 regular ideal tetrahedra and 2 regular ideal octahedra. An image of the tiling restricted to the boundary of horoball at $\infty$ of height $1$ is shown in \Cref{fig:Boyd}. Here the meridian translates by $-\bar{\omega}+i = \frac{-1+i(\sqrt{3}+2)}{2}$, a distance of $\sqrt{2+\sqrt{3}}$; compare \Cref{neilolarry} below.

    \begin{figure}[ht]
    \centering
    \includegraphics[height=4cm]{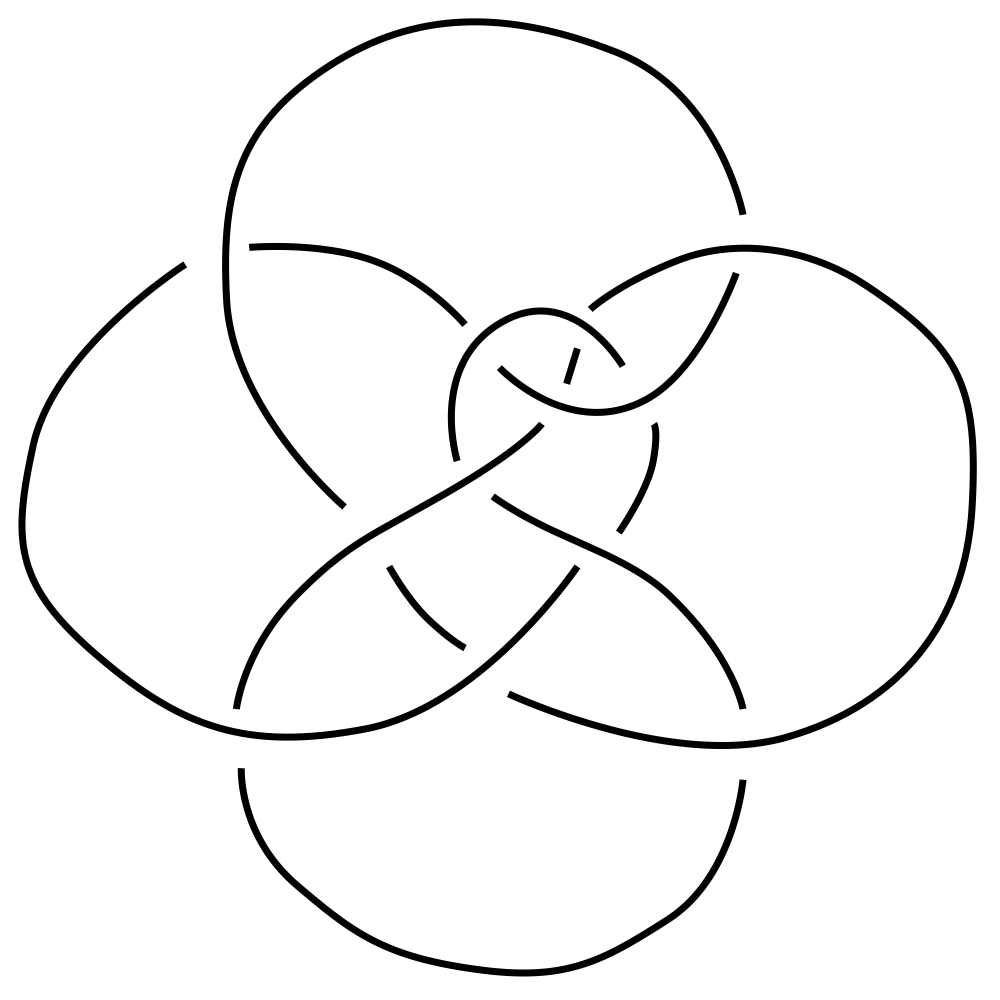}
    \hspace{.5cm}
    \includegraphics[height=4cm]{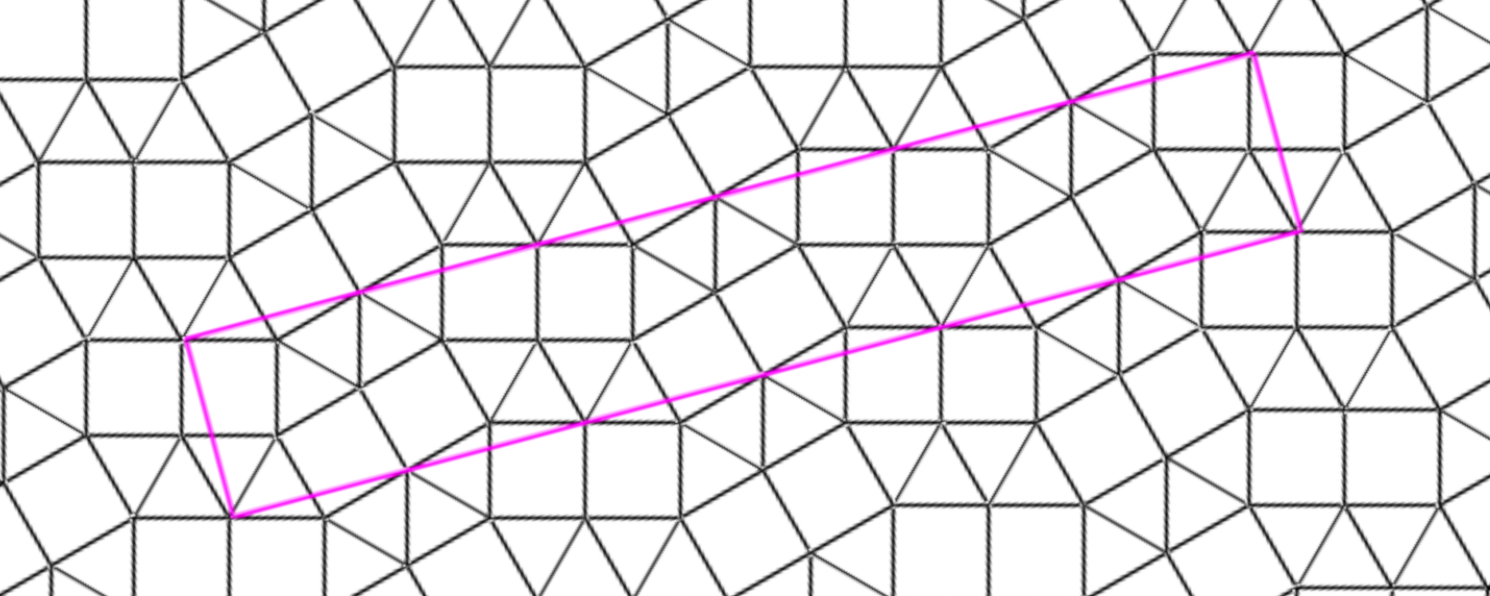}
    \caption{Knot $12n706$ and the peripheral tiling of its complement from SnapPy \cite{SnapPy} rotated to be in a standard position.}
    \label{fig:Boyd}
\end{figure}

    A notable structural feature of $M_B := S^3 - 12n706$ is visible in the fundamental domain for the knot complement's peripheral subgroup that is outlined in pink on the right-hand side of \Cref{fig:Boyd}: it is a rectangle with rationally related sidelengths. This implies that its \textit{cusp field} is $\mathbb{Q}(i)$, smaller than the \textit{shape field} generated by the tetrahedral parameters determined by its polyhedral decomposition (see \Cref{non arithmetic} below). 
    
    $M_B$ does not itself have hidden symmetries---this can be seen directly from the fact that the horoball tiling in \Cref{fig:Boyd} lacks the requisite rotational symmetries, or as a consequence of Theorem 1.1 of \cite{NeilOrbiCusps}. But the fact that its cusp field is smaller than its shape field---a very rare property among census knot complements, and likely a necessary one for the existence of hidden symmetries---suggests that mixed-platonic manifolds are worth further study vis a vis hidden symmetries.
 \end{example}   

Our next example shows that a mixed-platonic manifold can cover a one-cusped, rigid-cusped orbifold. Its features described here can be confirmed directly using SnapPy \cite{SnapPy}. 

\begin{example}\label{s913}
The two-cusped manifold $s913$ in SnapPy's \texttt{OrientableCuspedCensus} is mixed-platonic, decomposing as the union of one regular ideal octahedron and two regular ideal tetrahedra. Subdividing the octahedron produces a decomposition of $s913$ into six tetrahedra. It is the only census manifold with eight or fewer tetrahedra that is mixed-platonic: in fact it and $v2774$, which we will discuss further in \Cref{v2774}, are the only such manifolds whose volumes decompose as a positive integer linear combination of those of both the regular ideal tetrahedron and octahedron. We have checked this rigorously, running the verified computations of SnapPy within Sage \cite{sagemath} (script included as an ancillary file).

(Among the nine-tetrahedron census manifolds, at least  \texttt{o9\_44005}, \texttt{o9\_44006}, \texttt{o9\_44010}, and \texttt{o9\_44012} have mixed platonic decompositions. Interestingly, for each pair of manifolds in this list, there is a four-fold cover common to both indicating that these four manifolds belong to the same commensurability class.)  

The {orientation-preserving} symmetry group of $s913$ contains an order-four rotation whose axis of symmetry runs out of a cusp, and another which exchanges the two cusps. Therefore the quotient $O'$ of $s913$ by its orientation-preserving symmetry group is a one-cusped, rigid-cusped orbifold with a $(2,4,4)$ cusp.

\begin{figure}[ht]
\centering
    \includegraphics[height=6cm]{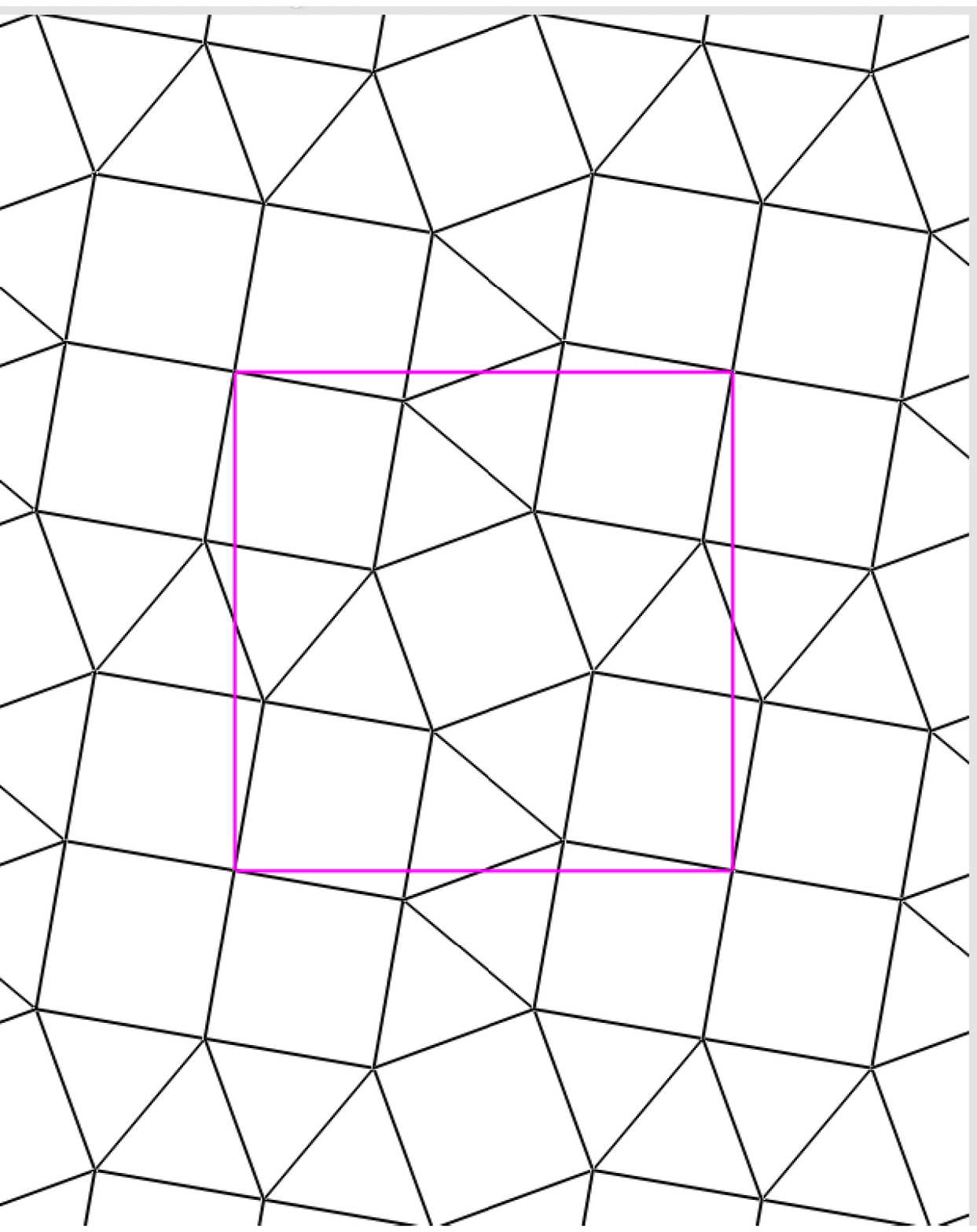}
    \hspace{.5cm}
    \includegraphics[height=6cm]{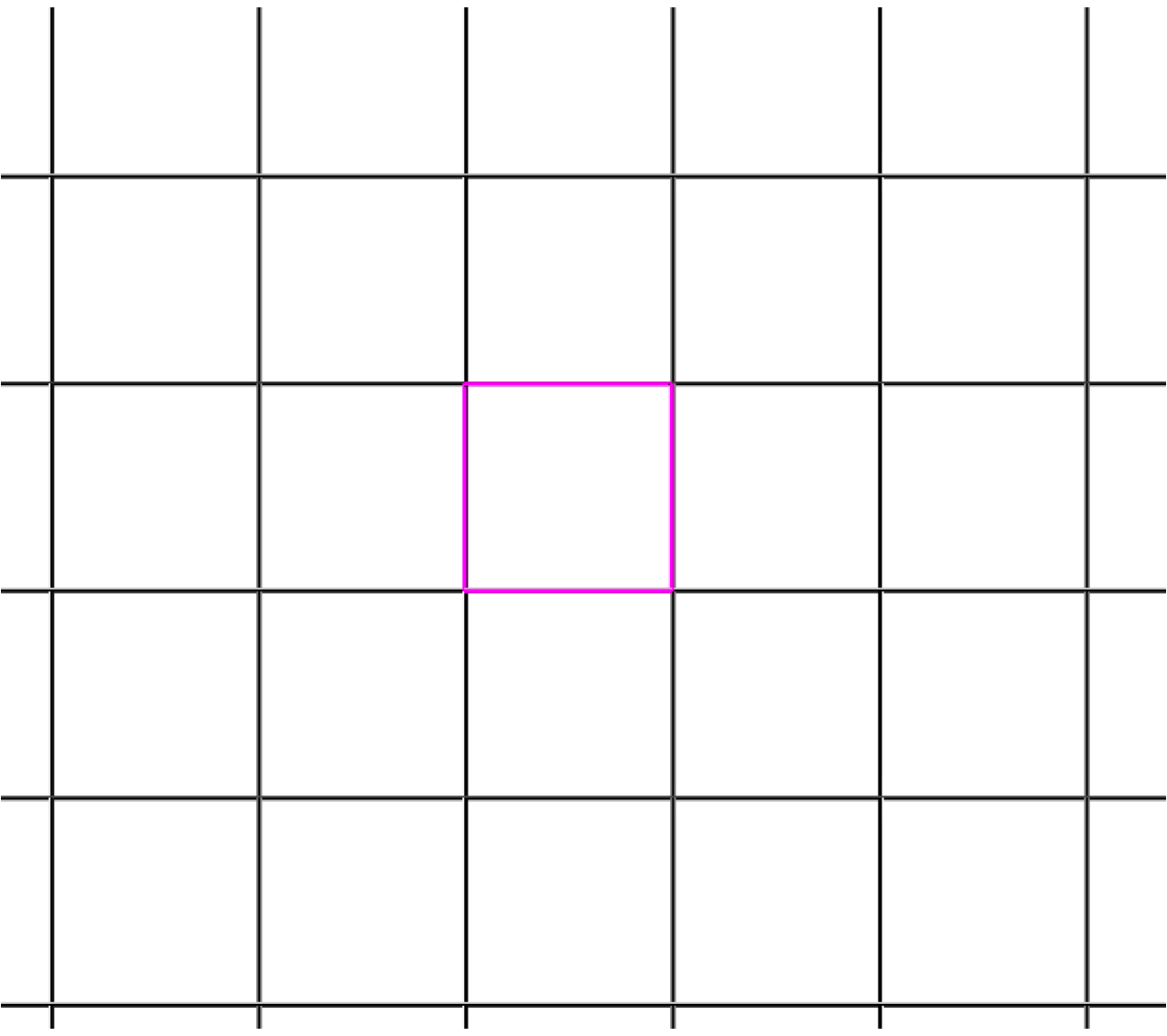}
\caption{\label{fig:s913} The two cusps of $s913$ induced by a mixed-platonic decomposition. Left: One cusp is tessellated by triangles and squares. Right: The other cusp is tessellated only by squares.} 
\end{figure}

As can be seen from horoball cross-sections of the two cusps of $s913$ (see \Cref{fig:s913}), no cusp-exchanging symmetry preserves its decomposition into an octahedron and tetrahedra. Therefore $O'$ does not inherit the structure of a mixed-platonic orbifold from this decomposition of $s913$. Other interesting features of $s913$ are discussed in \Cref{v2774}. 
\end{example}

\section{Basic Results}\label{basic}

In the arguments below, it will often be convenient to place the tiling of $\mathbb{H}^3$ by regular ideal tetrahedra and octahedra associated to a mixed-platonic orbifold in a ``best'' form. This is the \textit{standard position} of our first definition.

\begin{definition}\label{canon pos} Using the upper half-space model for $\Hthree$, we say that the regular ideal tetrahedron is in \textit{standard position} if its ideal vertices are at $0$, $1$, $\infty$ and $\omega = \frac{1}{2}(1+i\sqrt{3})$; and that the regular ideal octahedron is in \textit{standard position} if its ideal vertices are at $0$, $1$, $\infty$, $1+i$, $i$, and $\frac{1}{2}(1+i)$. For a regular ideal tetrahedron or octahedron in standard position, we name its face $\Delta$ that has ideal vertices at $0$, $1$, and $\infty$.

A tiling of $\mathbb{H}^3$ by regular ideal tetrahedra and/or octahedra is in \textit{standard position} if it has a tile in standard position, and we say that a mixed-platonic orbifold $\Hthree/\Gamma$ is \textit{represented in standard position} if $\Gamma$ leaves invariant a tiling of $\Hthree$ by regular ideal tetrahedra and octahedra that is in standard position.
\end{definition}

We begin with a lemma that shows that mixed-platonic hyperbolic manifolds and orbifolds are non-arithmetic. This lemma draws out an important feature of the definition of mixed-platonic manifolds in that there are a non-zero number of octahedra and a non-zero number of tetrahedra in any tiling corresponding to a mixed-platonic manifold.

\begin{lemma}\label{non arithmetic} The invariant trace field of a mixed-platonic hyperbolic $3$-manifold is $\mathbb{Q}(\zeta_{12})=\mathbb{Q}(\omega,i)$ where $\zeta_{12}$ is a $12$th root of unity and $\omega=\frac{1+\sqrt{-3}}{2}$. In particular, mixed-platonic hyperbolic $3$-manifolds are non-arithmetic.
\end{lemma}

\begin{proof}
    A necessary condition for a finite-volume hyperbolic $3$-manifold to be arithmetic is that its invariant trace field has exactly one complex place (see \cite{book}).
    Neumann and Reid show in \cite[Theorem 2.4]{NeumannReidArith} that the invariant trace field is equal to the number field generated by the shape parameters of a decomposition into ideal tetrahedra. The \textit{shape parameter} of an ideal tetrahedron with vertices at $0$, $1$, $\infty$, and some $z$ in the upper half-plane is defined to be $z$; it is known that this can always be arranged by an isometry, and that $z$ determines the tetrahedron up to isometry. 
    
    Viewing the regular ideal tetrahedron $T$ and octahedron $O$ in standard position, one sees directly that $T$ has shape parameter $\omega$, and that $O$ divides into four isometric ideal tetrahedra sharing an edge with ideal points $\infty$ and $\frac{1}{2}(1+i)$. One, and therefore all, of these tetrahedra is readily seen to have shape parameter $\frac{1}{2}(1+i)$. A mixed-platonic manifold $M = \mathbb{H}^3/\Gamma$ inherits a decomposition into regular ideal tetrahedra and octahedra projected from the $\Gamma$-invariant tiling of $\mathbb{H}^3$; dividing the octahedra further as described above yields a decomposition into ideal tetrahedra with shape parameters $\omega$ and $\frac{1}{2}(1+i)$. Since both shapes must be represented in the decomposition, every mixed-platonic manifold has invariant trace field $\mathbb{Q}(i,\omega)$. This field has 2 complex places and hence mixed-platonic manifolds are not arithmetic.
\end{proof}

Note that for an arbitrary mixed-platonic orbifold $\Hthree/\Gamma$ with a tiling $\Pthree$ by tetrahedra and octahedra stabilized by $\Gamma$, there is an orientation-preserving isometry $\phi$ such that $\phi(\Pthree)$ has a face with vertices at $0$, $1$, and $\infty$, and hence $\phi\Gamma\phi^{-1}$ is represented in standard position. 

\begin{prop}\label{glop} For any orientable mixed-platonic orbifold $\Hthree/\Gamma$ represented in standard position we have $\Gamma < \operatorname{PSL}_2(\ZZtwelve)$.
\end{prop}

\begin{proof} Let $\Pthree$ be a $\Gamma$-invariant tiling of $\Hthree$ by regular ideal tetrahedra and octahedra, in standard position. We will show that with this stipulated, $\Gamma< \operatorname{PSL}_2(\ZZtwelve)$.

This follows by an induction argument from two well known facts (see eg.~\cite{Hatcher_arithlinks}).\begin{itemize}
	\item $\operatorname{PSL}_2(\mathbb{Z}[\omega])$ preserves a tiling of $\Hthree$ by regular ideal tetrahedra, acting transitively on it. The tetrahedron $T$ in standard position belongs to this tiling, and its stabilizer in $\operatorname{PSL}_2(\mathbb{Z}[\omega])$ acts as the full orientation-perserving combinatorial symmetry group.
	\item $\operatorname{PSL}_2(\mathbb{Z}[i])$ preserves a tiling of $\Hthree$ by regular ideal octahedra, acting transitively on it. The octahedron $O$ in standard position belongs to this tiling, and its stabilizer in $\operatorname{PSL}_2(\mathbb{Z}[i])$ acts as the full orientation-perserving combinatorial symmetry group.\end{itemize}

Note that both $\mathbb{Z}[\omega]$ and $\mathbb{Z}[i]$ are subrings of $\mathbb{Z}(\zeta_{12})$, and hence both groups mentioned above are subgroups of $\operatorname{PSL}_2(\ZZtwelve)$. 

Now let $S_0$ be the tile of $\Pthree$ in standard position, let $\Delta$ be its face with ideal vertices $0$, $1$, and $\infty$, and suppose $\{(S_1,F_1),\hdots,(S_n,F_n)\}$ is a sequence with the property that for each $i$, $S_i$ is a tile of $\Pthree$ and $F_i = S_i\cap S_{i-1}$ is a triangular face. We claim that there is an element $\gamma$ of $\operatorname{PSL}_2(\ZZtwelve)$ with the property that $\gamma_0(S_n) = T$ or $O$ (according on whether $S_n$ is a tetrahedron or octahedron), and $\gamma_0(F_n) = \Delta$.

The claim is proved by induction on $n$. For $n = 1$ we take $\gamma = g^{-1}$, where $g$ in $\operatorname{PSL}_2(\mathbb{Z}[\omega])$ or $\operatorname{PSL}_2(\mathbb{Z}[i])$ (according to whether $S_0$ is a tetrahedron or octahedron) has the property that $g(S_0)\cap S_0 = F_1$. For $n>1$, assuming the claim for $n-1$, we take $\gamma_0\in \operatorname{PSL}_2(\ZZtwelve)$ with the property that $\gamma_0(S_{n-1}) = T$ or $O$ and $\gamma_0(F_{n-1}) = \Delta$. We then take $\gamma = g^{-1}\gamma_0$, where $g$ in $\operatorname{PSL}_2(\mathbb{Z}[\omega])$ or $\operatorname{PSL}_2(\mathbb{Z}[i])$ (according to whether $S_{n-1}$ is a tetrahedron or octahedron) has the property that $g(\gamma_0(S_{n-1}))$ intersects $\gamma_0(S_{n-1})$ along $\gamma_0(F_n)$. The claim follows.

Now for a fixed element $\gamma$ of $\Gamma$, choose a sequence $\{(S_1,F_1),\hdots,(S_n,F_n)\}$ as in the claim with the property that $S_n = \gamma(S_0)$. The claim supplies an element $\gamma_0$ of $\operatorname{PSL}_2(\ZZtwelve)$ taking $S_n$ back to $S_0$, so $\gamma_0\circ\gamma$ stabilizes $S_0$. There is an element $g_0$ of the stabilizer of $S_0$ in $\operatorname{PSL}_2(\mathbb{Z}[\omega])$ or $\operatorname{PSL}_2(\mathbb{Z}[i])$ taking $(\gamma_0\circ\gamma)(\Delta)$ back to its original position, ie.~such that $g_0\circ\gamma_0\circ\gamma$ fixes $0$, $1$, and $\infty$. We have thus exhibited $\gamma = (g_0\gamma_0)^{-1}$ as an element of $\operatorname{PSL}_2(\ZZtwelve)$.
\end{proof}

In what follows, a collection $\mathcal{H}$ of horoballs of $\mathbb{H}^3$ is a \textit{packing} if any two distinct members of $\mathcal{H}$ have disjoint interiors. A horoball packing $\mathcal{H}$ is \textit{maximal} if for any horoball packing $\mathcal{H}'$ such that each horoball of $\mathcal{H}$ is contained in one of $\mathcal{H}'$, $\mathcal{H}'=\mathcal{H}$. It is \textit{invariant under the action of} $\Gamma$, a group of isometries, if for every $\gamma\in\Gamma$ and $B\in\mathcal{H}$, $\gamma(B)\in\mathcal{H}$.

\begin{definition}\label{deaf and itchin'} For a regular ideal tetrahedron or octahedron $P$, let $\mathcal{H}_P$ be the unique maximal horoball packing by horoballs centered at its ideal vertices that is invariant under the action of the full self-isometry group of $P$.
\end{definition}

\begin{remark}\label{bort} If $P$ is a regular ideal tetrahedron in standard position as in \Cref{canon pos} then for each ideal vertex $p\in\{0,1,\omega\}$, the horoball $B_p$ of $\mathcal{H}_P$ centered at $p$ is the intersection with $\mathbb{H}^3$ of a Euclidean ball of radius $1/2$ that is tangent to $\mathbb{C}$ at $p$. The horoball $B_{\infty}$ centered at the final ideal vertex $\infty$ of $P$ is:
\begin{align}\label{my son is also named bort} 
    B_{\infty} = \{(z,t)\,|\,z\in\mathbb{C}, t\ge 1\}.
\end{align}
One can check this directly, for instance by observing that it is invariant under the order-three rotation fixing $\infty$ and exchanging the other three vertices, and under the rotation fixing $\omega$ and cyclically permuting $0$, $1$, and $\infty$.

If $P$ is a regular ideal octahedron in standard position then again $B_{\infty}$ is as above, and for $p\in\{0,1,1+i,i\}$ the horoball $B_p$ is again contained in a Euclidean ball of radius $1/2$ tangent to $\mathbb{C}$ at $p$. For $p = \frac{1+i}{2}$, $B_p$ is the intersection with $\mathbb{H}^3$ of a Euclidean ball of radius $1/4$ that is tangent to $\mathbb{C}$ at $p$. This can again be checked directly.
\end{remark}

\begin{lemma}\label{packing} Suppose $O = \Hthree/\Gamma$ is a mixed-platonic orbifold, and that $\Pthree$ is a $\Gamma$-invariant tiling of $\Hthree$ by regular ideal octahedra and regular ideal tetrahedra. There is a unique $\Gamma$-invariant horoball packing $\mathcal{H}$ of $\Hthree$ such that (1) the set of centers of horoballs of $\mathcal{H}$ is the set of ideal vertices of tiles of $\Pthree$; and (2) for each tile $P$ of $\Pthree$ and each ideal vertex $p$ of $P$, the horoball $B_p$ of $\mathcal{H}$ centered at $p$ belongs to $\mathcal{H}_P$ from \Cref{deaf and itchin'}. Each such horoball $B_p$ intersects a tile $P$ of $\Pthree$ if and only if $P$ has $p$ as an ideal vertex.

If $\Gamma$ is represented in standard position in the sense of \Cref{canon pos}, in the upper half-space model for $\Hthree$, then $\mathcal{H}$ contains a horoball $B_{\infty}$ centered at $\infty$ from (\ref{my son is also named bort}).

Conversely, $\mathcal{P}$ is the unique tiling by ideal polyhedra with edge set equal to the collection of geodesics joining the centers of tangent horoballs of $\mathcal{H}$ and face set equal to the collection of convex hulls of such geodesics corresponding to triples of pairwise-tangent horoballs of $\mathcal{H}$.\end{lemma}

\begin{remark}\label{canon} We will show in \Cref{that's amore} that $\Pthree$ is in fact the \emph{canonical cell decomposition} (aka \emph{Epstein-Penner decomposition}) of $\mathcal{H}$, in the sense established by Epstein-Penner \cite{EpstePen}.\end{remark}

\begin{proof}
Conditions (1) and (2) of the statement prescribe that we must have:
\[ \mathcal{H}=\bigcup_{P \in \Pthree} \mathcal{H}_P, \]
where $\mathcal{H}_P$ is as in \Cref{deaf and itchin'} for each $P\in\Pthree$. 
Note that if $\Gamma$ is represented in standard position then some $P\in\Pthree$ is in standard position and hence $B_{\infty}$ from (\ref{my son is also named bort}) belongs to $\mathcal{H}_P$ as described in \Cref{bort}. Note also that since $\Pthree$ is $\Gamma$-invariant, for any $\gamma \in \Gamma$, $\gamma$ sends $\mathcal{H}_P$ to $\mathcal{H}_{\gamma(P)}$ for each $P \in \Pthree$. Therefore $\mathcal{H}$ is also $\Gamma$-invariant.

We now observe that for any distinct $P_0$ and $P_1$ in $\Pthree$ intersecting along a face $f$, the collections of horoballs belonging to $\mathcal{H}_{P_0}$ and $\mathcal{H}_{P_1}$ and centered at ideal points of $f$ are identical. This is because their intersections with $f$ comprise a set of maximal horodisks invariant under the symmetry group of $f$, which is an ideal triangle, hence with each disk bounded by a horocyclic arc of length $1$. For any distinct $P$, $P'\in\Pthree$ sharing an ideal point $p$, there is a sequence $P = P_0,\hdots,P_n = P'$ such that $P_i$ shares a face with $P_{i-1}$ for each $i>0$. The observation above thus implies that for each ideal point $p$ of $\Pthree$, $\mathcal{H}$ has a unique horoball centered at $p$.

We next claim that for any horoball $B$ of $\mathcal{H}$ centered at an ideal point $p$ of $\mathbb{H}^3$, $B$ intersects only horoballs of $\mathcal{H}$ centered at  ideal vertices of tiles of $\Pthree$ with $p$ as an ideal vertex. For any such tile $P$, the three or four faces of $P$ with $p$ as an ideal vertex (in the respective cases that $P$ is a tetrahedron or octahedron) are each contained in a hyperplane of $\mathbb{H}^3$ that bounds a half-space containing $P$. $B$ does not intersect any face of $P$ without an ideal vertex at $p$---this can be checked directly by putting $P$ in standard position with $B = B_{\infty}$---so its intersection with the intersection of these three or four half-spaces is equal to its intersection with $P$. Therefore $B$ is contained in the interior of the union of the tiles of $\Pthree$ having $p$ as an ideal endpoint. The assertion of the final sentence of the Lemma's first paragraph follows.

Since this is also true for any other horoball $B'$, $B$ may intersect $B'$ only if there exists $P\in\Pthree$ having both horoballs' centers as ideal vertices, proving the claim.

The fact that each $\mathcal{H}_P$ is a maximal packing now implies that $\mathcal{H}$ is as well. In particular, from the second claim above and the explicit descriptions in \Cref{bort}, we find that distinct horoballs $B$ and $B'$ of $\mathcal{H}$ intersect if and only if they are centered at ideal points of an edge of some $P\in\Pthree$; and if so, that they are tangent at a single point of this edge. The edge set of $\Pthree$ is therefore characterized strictly in terms of $\mathcal{H}$ as the set of geodesics joining ideal points of distinct members that have a point of tangency. 

We similarly note that each triangular face of $\Pthree$ has a set of three distinct, pairwise-tangent horoballs centered at its ideal points (again this can be seen directly from the descriptions of \Cref{bort}). Conversely, for any such triple $\{B_1,B_2,B_3\}$, we claim that the set of centers of the $B_i$ is the set of ideal points of a face of $\Pthree$. To see this, translate $B_3$ to $B_{\infty}$, whereupon $B_1$ and $B_2$ are taken to Euclidean balls of radius $1/2$ whose centers are at distance $1$ due to their tangency. Let $e$ be the edge of $\Pthree$ joining the center $p$ of $B_1$ to $\infty$ and let $P_1,\hdots,P_k$ be the set of tiles of $\Pthree$ containing $e$, enumerated in cyclic order so that for each $i>1$, $P_i\cap P_{i-1}$ is a face $f_i$ containing $e$, and similarly $P_1\cap P_k$ is a face $f_1$. For each such $i$, $f_i$ has one ideal point at $\infty$, one at $p$ (the center of $B_1$), and a third---call it $p_i$---at distance $1$ from $p$ on $\mathbb{C}$. The claim is equivalent to asserting that the center $p'$ of $B_2$ is one such $p_i$.

As we observed above, $p'$ is an endpoint of an edge $e'$ of $\Pthree$ with its other endpoint at $\infty$. The point $e'\cap\partial B_{\infty}$ is a vertex of a tiling $\mathcal{T}$ of $\partial B_{\infty}$ by equilateral triangles and squares of sidelength $1$ that is obtained by intersecting it with the tiles of $\Pthree$ that have an ideal vertex at $\infty$. Each vertex of this induced tiling of $\partial B_{\infty} = \mathbb{C}\times\{1\}$ lies over an ideal endpoint of an edge of $\Pthree$ with its other endpoint at $\infty$; in particular, there are vertices over $p$ and each of the $p_i$ described above. For any such $i$, if $P_i\cap\partial B_{\infty}$ is a square (ie.~if $P_i$ is an octahedron) then its interior contains the entire open ninety-degree arc of the circle of radius $1$ centered at the vertex above $p$ that is bounded by those above $p_i$ and $p_{i+1}$. If $T := P_i\cap\partial B_{\infty}$ is a triangle then the corresponding open, sixty-degree circle arc is contained in the interior of the tile of $\mathcal{T}$ that intersects $T$ along its edge containing the vertices over $p_i$ and $p_{i+1}$. Thus in neither case does the open arc contain any vertex of $\mathcal{T}$, and it follows that the vertices above the $p_i$ are the only ones at distance one from the vertex over $p$. This implies the claim.

We have showed that the edge set of $\Pthree$ is the collection of geodesics joining the centers of pairs of tangent horoballs of $\mathcal{H}$, and the face set is the collection of ideal triangles determined by triples of pairwise-tangent horoballs. Therefore $\Pthree$ is determined by $\mathcal{H}$, since each of its cells is the closure of a complementary region to the union of the faces. 
\end{proof}

For our next result, recall that the \textit{commensurator} of a discrete group of isometries $\Gamma$ of $\Hthree$ is
\[ \mathrm{Comm}(\Gamma) = \{\gamma\in\mathrm{Isom}(\Hthree)\,|\, [\Gamma:(\gamma\Gamma\gamma^{-1})\cap\Gamma] <\infty\} \]
See eg.~\cite[p.~269]{book}. 

\begin{prop}\label{commensurator} Suppose $M = \Hthree/\Gamma$ is a one-cusped mixed-platonic manifold. The commensurator of $\Gamma$ preserves the $\Gamma$-invariant horoball packing $\mathcal{H}$ supplied by \Cref{packing}, hence also the associated $\Gamma$-invariant tiling $\mathcal{P}$ by regular ideal tetrahedra and octahedra.\end{prop}

\begin{remark} Thus in this case, $\mathcal{P}$ is the \emph{truly canonical cell decomposition} of Goodman-Heard-Hodgson \cite{GoodmanHeardHodgson}; cf.~\Cref{canon}.
\end{remark}

\begin{proof} In this proof it is useful to keep in mind that there is a correspondence between $\Gamma$-invariant horoball packings and horospherical cross-sections of $\Hthree/\Gamma$: the former project to the latter, and the latter lift to the former, under the universal covering map $\Hthree\to\Hthree/\Gamma$.

By \Cref{non arithmetic}, $M$ is non-arithmetic. Therefore by Margulis's arithmeticity theorem (see \cite[Theorem.~10.3.5]{book}), $\mathrm{Comm}(\Gamma)$ is discrete and a finite extension of $\Gamma$. We can therefore find a $\mathrm{Comm}(\Gamma)$-invariant horoball packing of $\Hthree$, which is thus also $\Gamma$-invariant. But since $M$ has only one cusp, there is only one $\Gamma$-invariant horoball packing up to $\Gamma$-equivariant rescaling (equivalently, any two horospherical cusp cross-sections are rescalings). 

It thus follows that $\mathcal{H}$ is a rescaling of the original $\mathrm{Comm}(\Gamma)$-invariant packing and hence is itself $\mathrm{Comm}(\Gamma)$-invariant. The fact that $\mathcal{P}$ is also $\mathrm{Comm}(\Gamma)$-invariant now follows from the final assertion of \Cref{packing}, since it is determined by $\mathcal{H}$.
\end{proof}

\begin{example}\label{v2774} The one-cusped manifold $v2774$ from SnapPy's orientable cusped census shares a double cover with the mixed-platonic manifold $s913$ considered in \Cref{s913}. Importantly, the Epstein-Penner decomposition for $v2774$ is a \emph{triangulation}\footnote{This triangulation has isomorphism signature {\tt ovLLwzQAPQcdgejighlkmnmnnmoafaofoqfhahxjo}.}, so by \Cref{commensurator}, $v2774$ is not mixed-platonic. 

We consider the consequences of this observation for the orientable commensurator quotient $O_0$ covered by $s913$ and $v2774$. (Note that since $s913$ is mixed-platonic, by \Cref{non arithmetic} it is not arithmetic and the commensurator quotient is an orbifold.) As described by \cite{CoulsonGoodmanHodgsonNeumann}, $O_0$ is the quotient of $\mathbb{H}^3$ by the symmetry group of the horoball packing for $v2274$, which is \emph{not} consistent with the horoball packing for $s913$ associated to the decomposition into regular ideal octahedra and tetrahedra. Therefore $O_0$ is also not mixed-platonic, since it has a unique horoball packing up to rescaling.
\end{example}

\Cref{v2774} shows that there are mixed-platonic manifolds that cover orbifolds which are not mixed-platonic, so \Cref{commensurator} does not generalize to the multi-cusped case. Equivalently, per \cite{CoulsonGoodmanHodgsonNeumann}, the horoball packing corresponding to a mixed-platonic manifold is not necessarily the maximally symmetric horoball packing. 

More can be said when further restricting to knot complements among one-cusped manifolds. The following proposition uses the properties of the ring $\ZZtwelve$ to control the possible meridians of knot complements in standard position.

\begin{prop}\label{neilolarry} Suppose $M = \Hthree/\Gamma$ is a mixed-platonic knot complement.
Any peripheral element of $\Gamma$ representing a meridian of $M$ has translation length $2+\sqrt{3}$ or $\sqrt{2+\sqrt{3}}$ on the horoball of $\mathcal{H}$ that it stabilizes, where $\mathcal{H}$ is the horoball packing supplied by \Cref{packing}.
\end{prop}

We note that the argument given below in \Cref{unit_claim} is a version of one in \cite[Proof of Lemma 3.6]{ReidWalsh} (see also \cite[Lemma 3.2]{Neil_SmallKnots}). We include a full proof here for the convenience of the reader. 

\begin{proof}
In the proof we will assume that $\Gamma$ is represented in standard position, and that the peripheral element in question belongs to the subgroup $\Lambda$ of $\Gamma$ stabilizing $\infty$. It thus has the form $\mu_\infty=\left(\begin{smallmatrix} 1 & \mI \\ 0 & 1 \end{smallmatrix}\right)$, where $\mI$ is an element of  $\ZZtwelve$ by \Cref{glop}.

\begin{claim}\label{unit_claim} $\mI$ is a unit of $\ZZtwelve$.
\end{claim}

\begin{proof}[Proof of \Cref{unit_claim}]
Suppose not; then there is a prime $\mathcal{P}$ of $\ZZtwelve$ such that $\mI\in\mathcal{P}$. Reducing entries modulo this prime gives a homomorphism from $\operatorname{PSL}_2(\ZZtwelve)$ to the finite group $\operatorname{PSL}_2(\ZZtwelve/\mathcal{P})$ with a meridian $\gamma$ in its kernel.  Since $\Gamma$ is normally generated by the meridian $\gamma$, it lies entirely in the kernel of this homomorphism if it exists. 

However, as $\Gamma$ is in standard position, there is also a parabolic fixed point at $0$, and thus a meridian fixing $0$ which is conjugate to $\mu_\infty$ by an element $g = \left(\begin{smallmatrix} a & b \\ c & d \end{smallmatrix}\right)$:

that is $$\left(\begin{smallmatrix} 1 & 0 \\ \mZero & 1 \end{smallmatrix}\right)=\left(\begin{smallmatrix} a & b \\ c & d \end{smallmatrix}\right) \left(\begin{smallmatrix} 1& \mI \\ 0 & 1 \end{smallmatrix}\right)\left(\begin{smallmatrix} d & -b \\ -c & a \end{smallmatrix}\right) .$$ 

We have that $g(\infty) = 0$ so $a=0$. Thus, $g$ cannot be trivial in $\operatorname{PSL}_2(\ZZtwelve/\mathcal{P})$, a contradiction.
\end{proof}

By \Cref{packing}, since $\Gamma$ is represented in standard position the horoball $B_{\infty}$ of $\mathcal{H}$ stabilized by $\Lambda$ is at height $1$. Therefore the translation length of $\left(\begin{smallmatrix} 1 & \mI \\ 0 & 1 \end{smallmatrix}\right)$ on $\partial B_{\infty}$ is the complex modulus $|\mI|$ of $\mI$; this equals the  ``waist size'' of the cusp of $M$ as defined by Adams in \cite{Adams_WaistSize}.

We now appeal to results of \cite{Adams_WaistSize}. Lemma 2.5 there asserts that any cusp of any hyperbolic $3$-manifold has waist size at least $1$, and Theorem 3.1 of \cite{Adams_WaistSize} asserts that the figure-eight knot complement is the unique hyperbolic $3$-manifold having a cusp with waist size equal to $1$. Since the figure-eight knot complement is arithmetic, we obtain that $|\mu|>1$ by \Cref{non arithmetic}.

Finally, the Six Theorem \cite{Lackenby_sixes} and \cite{AgolSix} implies that $|\mI|\le 6$. Since $\mathbb{Q}(\zeta_{12})$ has 2 complex places, the unit group in $\ZZtwelve$ is of the form $\mathbb{Z}/12\mathbb{Z} \times \mathbb{Z}$ with a fundamental unit of $\omega+i$ (see \cite[Chapter 5]{MarcusBook} for example). Hence, we then find that the only possibilities are that  $\mI = \omega + i$ or $\mI = 2+\sqrt{3}$, up to multiplication by a power of $\zeta_{12}$. These have the moduli listed in the Proposition statement.
\end{proof}

\begin{remark}
As a result of the proof above, we can say more about the off-diagonal entry of a meridian fixing $\infty$. In particular, $\mI$ is a unit of $\ZZtwelve$ of the form $\zeta_{12}^k(\omega+i)$ or $\zeta_{12}^k(2+\sqrt{3})$ where $\zeta_{12}$ is a primitive $12$th root of unity and $\omega =\frac{1}{2} + \frac{\sqrt{-3}}{2}$.
\end{remark}

\begin{prop}\label{prop_no_omega_plus_i}
Suppose $N = \Hthree/\Gamma$ is a mixed-platonic orbifold, let $c$ be a parabolic fixed point of $\Gamma$ and let $\Lambda$ be the subgroup of $\Gamma$ stabilizing $c$, and let $\mathcal{H}$ be the horoball packing supplied by \Cref{packing}. If $\Lambda$ has 3-torsion, then the minimal translation length in the horoball $H_c$ based at $c$ in $\Lambda$ cannot equal $\sqrt{2+\sqrt{3}}$.
\end{prop}

\begin{proof}
Let $\Lambda_0 \subset \Lambda$ be the maximal abelian subgroup of $\Lambda$ and assume that $\Lambda_0$ is generated by two translations of length $\ell$. Then $\Lambda_0$ stabilizes a fundamental domain $D$ tessellated by $n_1>1$ equilateral triangles and $n_2 >1$ squares  of side length $1$ in $\partial H_c$. We stress that since $\Lambda_0$ is acting on tessellation of the plane of this form, each translation in $\Lambda_0$ translates triangles to triangles and squares to squares. Considering this decomposition, the area is of the form $n_1\frac{\sqrt{3}}{4}+n_2$, where $n_1$ and $n_2$ are integers. On the other hand, the area of $D$ is $\frac{\ell^2\sqrt{3}}{2}$. If $\ell=\sqrt{2+\sqrt{3}}$, then this area is $\sqrt{3}+\frac{3}{2}$. Since $\frac{3}{2} \not\in \mathbb{Z}$, there is no valid choice of for $n_1$ and $n_2$.
\end{proof}

The following lemma is a summary of \cite{AdamsTop90} with extra details that are relevant to our discussion here.

\begin{lemma}\label{lem_cusp_volume}
Let $N = \Hthree/\Gamma$ be a mixed-platonic orbifold with a single rigid cusp. Then the cusp volume of $N$ is at least
\begin{enumerate}
\item $\sqrt{3}/8$ if $N$ has a $S^2(2,3,6)$ cusp,
\item $\sqrt{3}/4$ if $N$ has a $S^2(3,3,3)$ cusp, and
\item $1/4$ if $N$ has a $S^2(2,4,4)$ cusp.
\end{enumerate}
\end{lemma}

\begin{proof}[Proof Sketch.]
First consider the case that $N$ has a $S^2(2,3,6)$ cusp. Then by \cite[Theorem 
3.2]{AdamsTop90}, if the cusp volume of $N$ is less than $\sqrt{3}/8$, then 
it is a unique orbifold with cusp volume $\sqrt{3}/24$, $\sqrt{3}/12$, 
$1/8$, $\sqrt{3}(3+\sqrt{5})/48$, $\sqrt{21}/24$. Orbifolds with the first 
three cusp volumes are discussed in \cite{NeumannReidNotes} and noted to be  
arithmetic. Although not explicitly stated in the Adams' paper, a close 
reading of his computation shows that if the cusp volume is $\sqrt{3}
(3+\sqrt{5})/48$, the orbifold is the tetrahedral orbifold 
$\Gamma(2,2,5,2,6,3)$ (see \cite[Section 4.7.1]{book} for notation), which 
has invariant trace field $\mathbb{Q}(\sqrt{5},\sqrt{-3})$. The orbifold with cusp volume $\sqrt{21}/24$. This orbifold is discussed in \cite[Lemma 5.1]{Neil_SmallKnots} where a representation of the fundamental group is given. We will use that notation. Using \cite[Lemma 3.5.9]{book} with a generating set of $t,r,t.\gamma$, we can compute the invariant trace field from that representation, which is $\mathbb{Q}(\sqrt{-3})$. 

Each orbifold with a $S^2(3,3,3)$ cusp on Adams' list in \cite[Corollary 4.1]{AdamsTop90} is the unique 2-fold cover of an orbifold with a $S^2(2,3,6)$ cusp and is covered be previous argument.

In the paragraph below \cite[Theorem 5.1]{AdamsTop90} and \cite{NeumannReidNotes}, it is discussed that the orbifolds with $S^2(2,4,4)$ and cusp volume 1/8 and $\sqrt{2}/8$ are both arithmetic. 
\end{proof}

We conclude this section with a lemma that shows there is only one possible translation length for a meridian of a knot complement that is mixed platonic and admits hidden symmetries.

\begin{lemma}\label{lem_one_poss_translation}
Suppose $M = \Hthree/\Gamma$ is a mixed-platonic knot complement with hidden symmetries.
Any peripheral element of $\Gamma$ representing a meridian of $M$ has translation length $2+\sqrt{3}$ on the horoball of $\mathcal{H}$ that it stabilizes, where $\mathcal{H}$ is the horoball packing supplied by \Cref{packing}.    
\end{lemma}

\begin{proof}
The assumption that $M$ has hidden symmetries implies that the orientable commensurator quotient $N=\Hthree/Comm^+(\Gamma)$ of $M$ has 3-torsion on the cusp by \cite[Theorem 1.1]{NeilOrbiCusps} and \cite[Prop 9.1]{NeumannReidArith}. Assume that $\Gamma$ has a meridian of length $\sqrt{2+\sqrt{3}}$. Then the cusp volume for $N$ is at most $\frac{(\sqrt{2+\sqrt{3}})^2\sqrt{3}}{12}$ if $N$ has $6$-torsion on the cusp and $\frac{(\sqrt{2+\sqrt{3}})^2\sqrt{3}}{6}$ if not. Then by \Cref{prop_no_omega_plus_i}, we have that the minimal translation length for has to be smaller than this. Thus, the cusp volume is at most $\frac{(\sqrt{2+\sqrt{3}})^2\sqrt{3}}{24} < \frac{\sqrt{3}}{8}$ in the 6-torsion case and $\frac{(\sqrt{2+\sqrt{3}})^2\sqrt{3}}{12} < \frac{\sqrt{3}}{4}$ in the other case. In either case, \Cref{lem_cusp_volume} shows that this is not possible.
\end{proof}

\section{Peripheral tilings}\label{perti}

In this section we define \emph{peripheral tilings} of a mixed-platonic manifold or orbifold and prove a series of basic results about them, allowing us to show that certain tilings by triangles and squares cannot be the peripheral tiling of any one-cusped mixed-platonic manifold.

\begin{definition}\label{peripheral} A tiling $\Ttwo$ of $\mathbb{C}$ is a \textit{peripheral tiling} of a mixed-platonic manifold or orbifold $M = \Hthree/\Gamma$ if 
$\Gamma$ can be represented in standard position in the sense of \Cref{canon pos}, with  $\Gamma$-invariant tiling $\Pthree$, so that $\Ttwo = \Pthree\cap\partial B_{\infty}$, where $B_{\infty}$ is the horoball centered at $\infty$ of the packing supplied by \Cref{packing}. \end{definition}

Peripheral tilings were considered above (not by name) in \Cref{packing}. As observed there, such a tiling will necessarily be by equilateral triangles and squares, all of sidelength $1$.

\begin{lemma}\label{on til} For a mixed platonic orbifold $O = \mathbb{H}^3/\Gamma$ and any fixed cusp $c$ of $O$, there is a unique peripheral tiling $\Ttwo$ of $O$ such that $\Ttwo = \Pthree\cap\partial B_{\infty}$, for $B_{\infty}$ as in \Cref{peripheral} projecting to a neighborhood of $c$.

Now suppose $c$ is the only cusp of $O$ and, representing $\Gamma$ in standard position as in \Cref{peripheral}, let $\Lambda$ be the stabilizer of $\infty$ in $\Gamma$. For $\Ttwo$ and $B_{\infty}$ as above, and for every tile $P$ of $\Pthree$ and every ideal vertex $v$ of $P$, taking $B_v$ to be the horoball of the packing supplied by \Cref{packing} that is centered at $v$, there is a unique $\Lambda$-equivalence class of tiles of $\mathcal{T}$ that are $\Gamma$-isometric to $P\cap\partial B_v$.
\end{lemma}

\begin{proof}
    With $\Gamma$ represented in standard position with associated tiling $\Pthree$, let $\mathcal{H}$ be the horoball packing supplied by \Cref{packing}. Each horoball $B\in\mathcal{H}$ projects to a cusp neighborhood in $O$, and the preimage of any such cusp cross-section is a $\Gamma$-orbit in $\mathcal{H}$. For a fixed cusp $c$ of $O$ and a horoball $B\in\mathcal{H}$ projecting to a neighborhood of $c$, there is an isometry $\beta$ of $\mathbb{H}^3$ taking $B$ to $B_{\infty}$ and putting a tile of $\Pthree$ with an ideal vertex at the center of $B$ into standard position. Thus replacing $\Pthree$ and $\mathcal{H}$ with their images under this isometry, and $\Gamma$ by $\delta\Gamma\delta^{-1}$, we have again represented $\Gamma$ in standard position in such a way that $B_{\infty}$ projects to $c$ and $\Ttwo=\Pthree\cap\partial B_{\infty}$ is the associated peripheral tiling.

    That $\mathcal{T}$ is uniquely associated to $c$ stems from the fact that if $\delta$ belongs to $\Gamma$ then it preserves the tiling $\Pthree$, and hence the replacement procedure above does not change $\mathcal{T}$. And $B$ and $B_{\infty}$ project to the same cusp of $O$ if and only if they are $\Gamma$-equivalent. If $O$ has only one cusp then every horoball of $\mathcal{H}$ is $\Gamma$-equivalent to $B_{\infty}$. In this case, this holds in particular for the horoball $B_v\in\mathcal{H}$ centered at a prescribed ideal vertex $v$ of a prescribed tile $P$ of $\Pthree$, so $\delta$ as above takes $P\cap\partial B_v$ to a tile $T$ of $\mathcal{T}$. For another element $\gamma\in\Gamma$ taking $B_v$ to $B_{\infty}$, $\gamma\delta^{-1}\in\Lambda$ takes $T$ to $\gamma(P\cap\partial B_v)$.
\end{proof}

It is not true that every tiling $\Ttwo$ of $\mathbb{C}$ is the peripheral tiling of a mixed-platonic manifold or orbifold. In this section we establish a set of necessary criteria for a tiling to be a peripheral tiling, specifically of a \textit{one-cusped} mixed-platonic manifold or orbifold, and exhibit certain tilings that fail them.

\begin{definition} For a peripheral tiling $\mathcal{T}$ of a mixed-platonic manifold or orbifold, we declare a triangular tile to be \textit{TTT} if it abuts three triangles of $\mathcal{T}$ along its edges, \textit{OTT} if it abuts a square and two triangles of $\mathcal{T}$, \textit{OOT} if it abuts two squares and a triangle, and \textit{OOO} if it abuts all squares.\end{definition}

\begin{prop}\label{cheese brother} 
Suppose a tiling $\mathcal{T}$ of $\mathbb{C}$ is a peripheral tiling of a one-cusped, mixed-platonic orbifold. If $\mathcal{T}$ contains triangles of type OOT then it also contains some of type OOO or OTT; and if $\mathcal{T}$ contains some of type OTT then it also contains some of type OOT or TTT.\end{prop}\label{triangular tile conditions}

\begin{proof} Suppose such a tiling $\mathcal{T}$ has triangular tiles of type OOT but none of type OTT. Consider one triangular tile of type OOT as in Left of \Cref{triangluar tile types}.

\begin{figure}[ht]
\begin{tikzpicture}[scale=0.5]

\draw [thick] (0,0) -- (2,0) -- (1,1.732) -- cycle;
\draw [thick] (0,0) -- (0,-2) -- (2,-2) -- (2,0) -- (3.732,1) -- (2.732,2.732) -- (1,1.732) -- (-1,1.732) -- cycle;

\draw (0,0) circle [radius=0.2];
\fill (2,0) circle [radius=0.2];

\begin{scope} [xshift=8cm]
\draw [thick] (-1,1.732) -- (1,1.732) --(3, 1.732)--(2,0)--(2,-2)--(0,-2)--(0,0)--(-1,1.732);
\draw [thick] (0,0)--(1,1.732)--(2,0)--cycle;
\draw (0,0) circle [radius=0.2];
\fill (1,1.732) circle [radius=0.2];

\end{scope}
\end{tikzpicture}

\caption{Left: Triangular tile of type OOT, Right: Triangular tile of type OTT}
\label{triangluar tile types}

\end{figure}
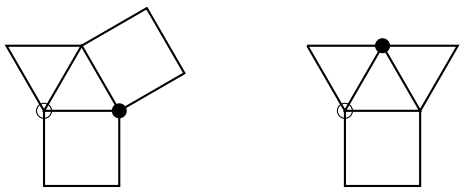

Each vertex of $\mathcal{T}$ represents an edge of the associated $\Gamma$-invariant tiling of $\Hthree$ that has one ideal point at $\infty$. Let $v_0$ be the other ideal endpoint of the edge represented by the open-circled vertex in the left of \Cref{triangluar tile types}, and let $P$ be the regular ideal tetrahedron that intersects $\mathcal{T}$ in the middle triangle (of type OOT). Of the two faces of $P$ containing both $v_0$ and $\infty$, one abuts an octahedron and one a tetrahedron. It follows from \Cref{on til} that there exists $\gamma \in \Gamma$ such that $\gamma(v_0)=\infty$ and $\gamma(P \cap \mathcal{T})$ is a tile in F. Now if the final face $f$ of $P$ containing $v_0$, the one opposite $\infty$, abuts a tetrahedron, then $\gamma(P \cap \mathcal{T})$ is a triangle of type OTT contradicting our assumption. So, $f$ abuts an octahedron. Let $v_1$ be the ideal vertex at the other endpoint of the edge represented by the shaded-circled vertex. Then all faces of $P$ that contain the ideal vertex $v_1$ abut octahedra. Thus using \Cref{on til} again, we can conclude that there is triangular tile of type OOO. 

An entirely analogous argument using the open circled and the shaded-circled vertices in the right of \Cref{triangluar tile types} shows that if $\mathcal{T}$ contains tiles of type OTT and none of type OOT then it also contains some of type TTT. Hence, we are done. \end{proof}

The result below is a sample application of \Cref{cheese brother}

\begin{corollary}\label{pom} Neither tiling of $\mathbb{C}$ consisting of $\Lambda_0$-translates of the fundamental domains of \Cref{tiddly}, where $\Lambda_0 = \langle \mu,\lambda \rangle$, is the peripheral tiling of a one-cusped, mixed-platonic orbifold.\end{corollary}

\begin{figure}[ht]
\begin{tikzpicture}

\begin{scope}[xshift=-4cm, yshift=1cm]
\fill[color=yellow!20] (0,0)--(2.366,-1.366)--(4.732,0)--(2.366,1.366)--cycle;
\draw [thick] (0,0) -- (1,0) -- (0.5,0.866) -- cycle;
\draw [thick, xshift=1.866cm, yshift=0.5cm] (0,0) -- (1,0) -- (0.5,0.866) -- cycle;
\draw [thick, xshift=3.732cm] (0,0) -- (1,0) -- (0.5,0.866) -- cycle;
\draw [thick, xshift=1.366cm, yshift=-1.366cm] (0,0) -- (1,0) -- (0.5,0.866) -- cycle;
\draw [thick, xshift=2.366cm, yshift=-1.366cm] (0,0) -- (1,0) -- (0.5,0.866) -- cycle;

\draw [thick] (0,0) -- (1,0) -- (0.5,-0.866) -- cycle;
\draw [thick, xshift=3.732cm] (0,0) -- (1,0) -- (0.5,-0.866) -- cycle;
\draw [thick, xshift=1.366cm, yshift=1.366cm] (0,0) -- (1,0) -- (0.5,-0.866) -- cycle;
\draw [thick, xshift=2.366cm, yshift=1.366cm] (0,0) -- (1,0) -- (0.5,-0.866) -- cycle;

\draw [thick] (1.866,0.5) -- (1.866,-0.5) -- (2.866,-0.5) -- (2.866,0.5);

\draw [thick] (1,0) -- (1.866,0.5) -- (1.366,1.366) -- (0.5,0.866) -- cycle;
\draw [thick, xshift=2.366cm, yshift=-1.366cm] (1,0) -- (1.866,0.5) -- (1.366,1.366) -- (0.5,0.866) -- cycle;

\draw [thick] (1,0) -- (1.866,-0.5) -- (1.366,-1.366) -- (0.5,-0.866) -- cycle;
\draw [thick, xshift=2.366cm, yshift=1.366cm] (1,0) -- (1.866,-0.5) -- (1.366,-1.366) -- (0.5,-0.866) -- cycle;

\draw [->] (0,0) --(1.183, -.683);
\draw (1.183, -.683)--(2.366,-1.366);
\draw [->] (0,0) -- (1.183, .683);
\draw (1.183, .683)--(2.366,1.366);
\draw [dashed] (2.366,1.366) -- (4.732,0) -- (2.366,-1.366);

\node at (.9,-0.85) {$\mu$};
\node at (.9,0.85) {$\lambda$};

\end{scope}

\begin{scope}[xshift=4cm]
\fill[color=yellow!20] (0,0)--(0,1.93)--(1.93,1.93)--(1.93,0)--cycle;
\draw [dashed] (0,0) -- (0,1.93) -- (1.93,1.93) -- (1.93,0) -- cycle;
\draw [thick] (0,0) -- (0.2588,0.9659) -- (0,1.93) -- (0.9659,1.673) -- (1.93,1.93) -- (2.1906,0.9659) -- (1.93,0) -- (0.9659,-0.2588) -- cycle;
\draw [thick] (0.9659,1.673) -- (0.2588,0.9659) -- (1.2247,0.7071) -- (0.9659,-0.2588);
\draw [thick] (0.9659,1.673) -- (1.2247,0.7071) -- (2.1906,0.9659);
\draw [thick] (1.2247,0.7071) -- (1.93,0);

\draw [->] (0,0) -- (0, .965);
\draw (0,.965)--(0,1.93); 
\draw [->] (0,0) -- (.965,0);
\draw(.965,0)--(1.93,0);
\node [left] at (0,0.9659) {$\lambda$};
\node [above] at (0.9,0) {$\mu$};

\end{scope}

\end{tikzpicture}
\caption{Tilings of $\mathbb{C}$ that cannot appear as peripheral tilings of mixed-platonic manifolds }
\label{tiddly}
\end{figure}
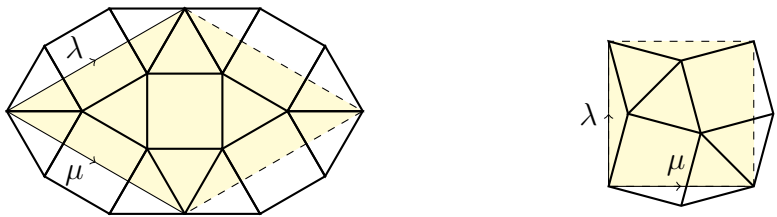

\begin{proof} Inspecting the fundamental domain on the left, we find only triangles of type OOO and OTT. On the right, every triangular tile is of type OOT. \Cref{cheese brother} 
now implies the result. \end{proof}

\begin{lemma}\label{one tile} For a one-cusped, mixed-platonic manifold $M=\Hthree/\Gamma$, upon representing $\Gamma$ in standard position with $\Ttwo\subset\partial B_{\infty}$, the peripheral subgroup $\Lambda<\Gamma$ stabilizing $\infty$ has a connected fundamental domain $F$ that is a union of tiles of $\Ttwo$, with the property that for every ideal vertex $v$ of every tile $P$ of the associated $\Gamma$-invariant tiling of $\Hthree$, there is a unique $\gamma\in\Gamma$ taking $v$ to $\infty$ such that $\gamma.P\cap\mathcal{T}$ is a tile of $F$. In particular, $F$ has exactly $4n_1$ triangular faces, and $6n_2$ octahedral faces, for $n_1$ and $n_2$ as in \Cref{mtothep}.\end{lemma}

\begin{proof}
The key observation here is that if $M=\Hthree/\Gamma$ has only one cusp then any two horoballs of the packing $\mathcal{H}$ supplied by \Cref{packing} are $\Gamma$-equivalent. By the definition of mixed-platonic orbifolds, the $\Gamma$-invariant tiling by tetrahedra and octahedra has at least one $\Gamma$-orbit of each type. Any ideal tetrahedron of this tiling has a horoball of $\mathcal{H}$ centered at each of its ideal points; hence a $\Gamma$-translate of this tetrahedron has an ideal vertex at $\infty$---and therefore intersects $\partial B_{\infty}$. The same holds for ideal octahedra.

If $M$ is a manifold then since $\Gamma$ is torsion-free, no element of the peripheral subgroup $\Lambda$ stabilizing $\mathcal{T}$ takes any square or triangle to itself. It is thus possible to construct a fundamental domain for $\Lambda$ as a union of squares and triangles, by starting with a given tile and working outward along edges, including each newly-encountered tile in the union if and only if its $\Lambda$-orbit does not belong to the union. This process terminates, since $\partial B_{\infty}/\Lambda$ has finite area. It yields a connected union $F$ of squares and triangles, containing at most one representative of each $\Lambda$-orbit, with the property that each tile edge in the frontier of $F$ is the intersection of a tile in $F$ and a tile outside $F$ that is $\Lambda$-equivalent to one inside of it. We claim that $F$ contains a representative of \textit{every} $\Lambda$-orbit and hence is the desired fundamental domain.

If this were not the case then by choosing a tile $P$ of $\mathcal{T}$ that is closest to $F$ with the property that its orbit is not represented in $F$, we can ensure that a  tiles $P'$ abutting $P$ does have its $\Lambda$-orbit represented in $F$. But then the element of $\Lambda$ that moves $P'$ into $F$ would take $P$ either into or adjacent to $F$, respectively contradicting either the selection of $P$ or the construction of $F$. This proves the claim.

In the case that $M$ is a manifold and hence $\Gamma$ is torsion-free, we note further that no element of $\Gamma$ stabilizes any tile of its invariant tiling. Since $M$ is one-cusped it follows that each $\Gamma$-orbit of tetrahedra (or respectively, octahedra) yields exactly four (resp.~six) distinct $\Lambda$-orbits of triangles (resp.~quadrilaterals) in $\mathcal{T}$.
\end{proof}

We conclude this section by using a more sophisticated version of the same approach to prove the following result, addressing a tiling that proves exceptional from other perspectives.

\begin{prop}\label{no ace} The tiling \aceofbase~ pictured in \Cref{fig_ace} 
is not the peripheral tiling of a one-cusped mixed-platonic manifold.\end{prop}

\begin{figure}
\includegraphics[width=3.5 in]{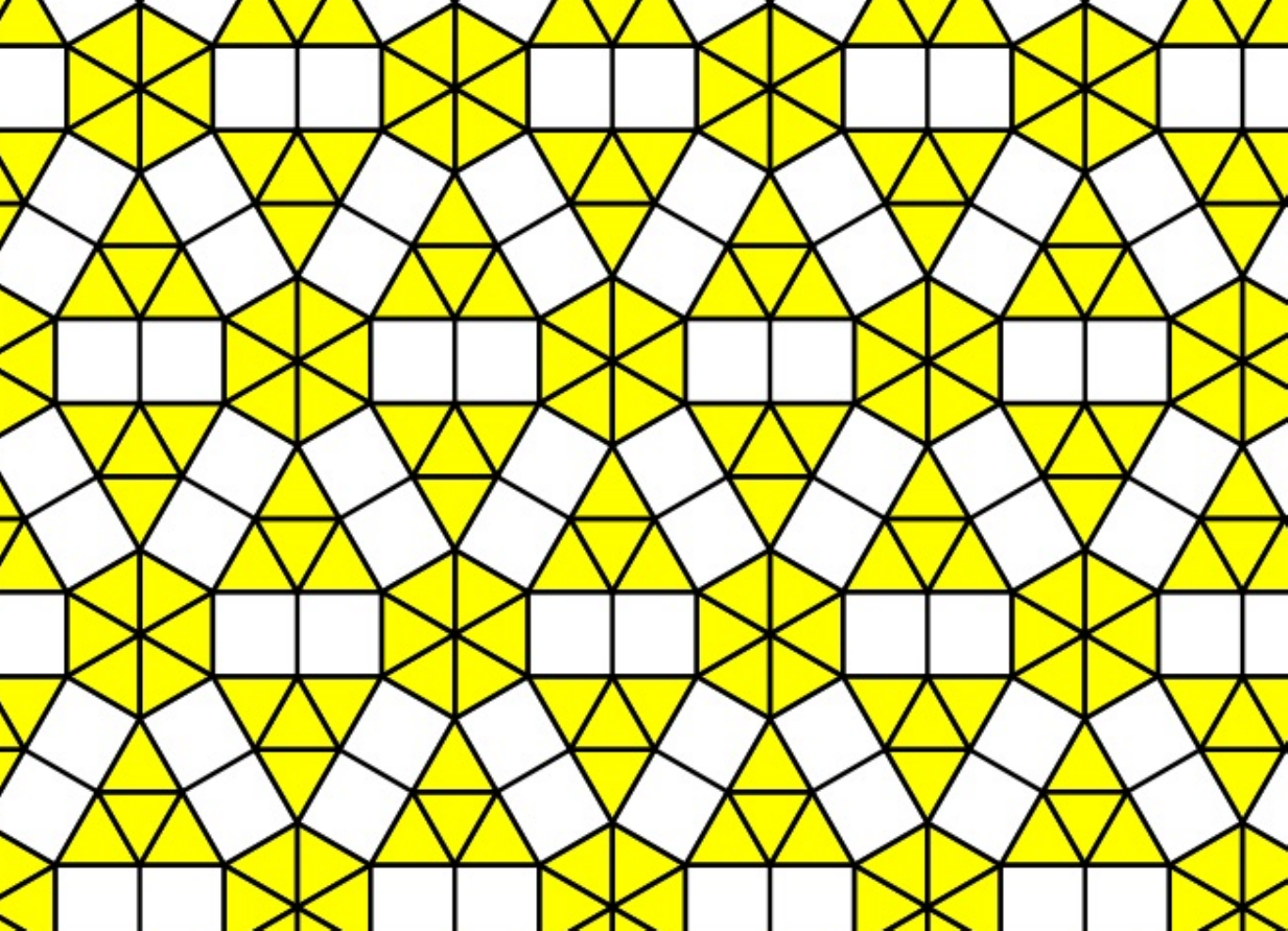}
\caption{\aceofbase~ - a tiling of the plane where each square is surrounded by three triangles and one other square, which also has symmetries of order $3$. }
\label{fig_ace}
\end{figure}

\begin{proof}
The orientation-preserving symmetry group $\Lambda_0$ of $\aceofbase$\ has an order-six rotation fixing the center of each hexagonal region that is a union of triangles of $\aceofbase$\ (yellow in the Figure); an order-three rotation fixing the center of each triangular region that is a union of triangles of $\aceofbase$\ (also yellow in the Figure), and an order-two rotation exchanging any pair of squares of $\aceofbase$\ that share an edge. Therefore it is a $(2,3,6)$-rotation group. It acts on the triangle classes as follows:\begin{itemize}
	\item transitively on TTT triangles---each of which is at the center of a yellow triangular region, and has a stabilizer of order three;
	\item simply transitively on the OTT triangles, which occupy the yellow hexagonal regions;
	\item simply transitively on the OOT triangles, each of which lies in a corner of a yellow triangular region; and
	\item simply transitively on squares.\end{itemize}
Consequentially, a fundamental domain for $\Lambda_0$ must contain the area of one square, OTT triangle, and OOT triangle each, and one-third the area of a TTT triangle. Therefore any lattice subgroup $\Lambda$ of $\Lambda_0$ that has $d$ distinct orbits of squares will also have $d$ distinct orbits each of OOT and OTT triangles, and $d/3$ distinct orbits of TTT triangles.

Now let us suppose that a one-cusped mixed-platonic manifold $M$ that divides into regular ideal tetrahedra and octahedra has the tiling of $\aceofbase$\ as its peripheral tiling. We classify the tetrahedra decomposing $M$ as follows:\begin{itemize}
	\item ``TTTT'', for those abutting tetrahedra across all faces;
	\item ``OOTT'', for those abutting octahedra across two faces and tetrahedra across the others; and
	\item ``OTTT'', for those abutting an octahedron across one face and tetrahedra across the others.\end{itemize}
Of the four triangular horospherical cross sections of the ideal vertices of a tetrahedron, for the types above we have the following classification:\begin{itemize}
	\item for a tetrahedron of type TTTT, all four triangles are of type TTT;
	\item for a tetrahedron of type OOTT, two triangles are of type OTT, and two are of type OOT;
	\item for a tetrahedron of type OTTT, three triangles are of type OTT, and one is of type TTT.\end{itemize}
Note that while there could in principle be other tetrahedron types (eg.~OOOT), these all have at least one vertex with a triangular horospherical cross section not represented in $\aceofbase$. So the decomposition of $M$ must only involve tetrahedra of types TTTT, OOTT, and OTTT. We take the numbers of these to be $a$, $b$ and $c$, respectively. 

Now suppose the decomposition of $M$ has $\sigma$ distinct octahedra. Since each octahedron has six ideal vertices, each with horospherical cross section a square, the peripheral tiling of $M$ has $d = 6\sigma$ squares. We thus find that:\begin{itemize}
	\item $d = 2b + 3c$, counting triangles of type OTT;
	\item $d = 2b$, counting triangles of type OOT; and
	\item $d = 3(4a + c)$, counting triangles of type TTT.\end{itemize}
The first two equations give that $c = 0$, whereupon we further obtain $a = d/12 = \sigma/2$ and $b = d/2 = 3\sigma$. In particular, $\sigma$ must be even.

But the only takeaway that we will use here is that every triangle of type OTT comes from a tetrahedron of type OOTT, since there are no tetrahedra of type OTTT (``$c=0$'' above). Consider a hexagon $H$ from the tiling of \Cref{fig_ace}, the union of six OTT triangles sharing a vertex. Each such triangle is the ideal vertex of an OOTT tetrahedron. For adjacent such tetrahedra $T_1$ and $T_2$, let $f$ be the triangular face that they share and let $v$ be the ideal vertex of $f$ that does not belong to the edge of $f$ with one endpoint at the rotation center of $H$. Let $B_v$ be the horoball of the packing from \Cref{packing} that is centered at $v$.

For each of $i=1,2$, the horospherical cross section $B_v\cap T_i$ is a triangle of type OOT in $B_v$, and the two cross-sectional triangles $B_v\cap T_1$ and $B_v\cap T_2$ meet each other across their ``T'' edges. But no such arrangement occurs in the tiling of \Cref{fig_ace}, where each OOT triangle intersects a TTT triangle across its ``T'' edge. This contradicts the hypothesis that $M$ is one-cusped, since every cross-sectional triangle of every polyhedron in its decomposition must lift to a tile of the Figure.
\end{proof}

\section{Classifying rotationally symmetric tilings with a short translation}\label{Pgh Regional Transit}

In this section, we classify tilings of $\mathbb{R}^2$ that are built from unit-sidelength squares and equilateral triangles, and have order-$3$ rotational symmetry and a translation of length at most $2+\sqrt{3}$. The section's main result \Cref{And then there were two} asserts that there are three of these, generated from fundamental domains in Figure \ref{dumb bell} as described in \Cref{embark}. We call them $\mathcal{Q}$ (which is also pictured on the left in Figure \ref{tiddly}), $\aceofbase$ (also in Figure \ref{fig_ace}), and $\totaleclipse$ (also in Figure \ref{fig_total_eclipse}). While the results here follow from classifications of more general plane tilings, we include a complete proof in the interests of keeping the paper self-contained. 

We first briefly relate our results to the general classification of plane tilings with a cocompact symmetry group. Such tilings are usually called \emph{$k$-uniform} and classified by $k$, defined to be the number of transitivity classes (ie.~orbits) of vertices under the action of the full symmetry group of the tiling. A close reading of our arguments below does lead to the conclusion that the tilings we are interested in appear in the list of $k$-uniform tilings for $k=2$ or $k=3$. The tiling $\mathcal{Q}$ is one of the twenty $2$-uniform tilings of the plane, and $\aceofbase$ and $\totaleclipse$ are each among the sixty-one $3$-uniform tilings of the plane (see \cite{MR0264511} for a classification of $2$-uniform tilings, \cite{MR0745905,chavey2014tilings} for a classification and relevant discussion of $3$-uniform tilings, and \cite{grunbaum1987tilings} for further background). In particular, these classifications often contain tilings that allow for tiles with more than four sides (e.g. hexagons and $12$-gons). Just as relevant, to apply those results we would also need show that the full symmetry groups contain both order $6$-rotations and reflections. While these conclusions can be inferred from the classification contained here, we present the argument this way because it more directly connects notations of translation length to an enumeration of valid peripheral tilings.

It also follows from this section that the tilings of the plane by equilateral triangles and squares with a rotation of order $3$ and shortest translation length \emph{equal to} $2+\sqrt{3}$ are exactly $\aceofbase$ (see \Cref{fig_ace})  and $\totaleclipse$ (see \Cref{fig_total_eclipse}). Alternative names for $\aceofbase$ include $[3^6;3^34^2;3^24.3.4]$.  
This notation classifies the cyclic orders of polygons incident to the three orbits of vertices: in this case, one vertex is incident to six triangles; a second, to three successive triangles followed by two squares; and the third to two triangles, one square, and a third triangle before being capped off by a final square. The notation for $\totaleclipse$ is $[3^6; 3^2 4.3.4; 3^2 4.3.4]$. 

The rest of the section leads up to the proof of \Cref{And then there were two}.

\begin{lemma}\label{chow-dah! It's chow-dah!}
Fix $t>0$ and suppose $\Lambda$ is a discrete, cocompact group of isometries of $\mathbb{R}^2$ containing rotations of order three, such that the minimal translation length of any infinite-order element of $\Lambda$ is $t$. Then:\begin{itemize}
    \item There is a tiling of $\mathbb{R}^2$ by equilateral triangles with sides of length $t/\sqrt{3}$, whose vertex set is the set of fixed points of order-three rotations in $\Lambda$. For any triangle $\Delta$ of this tiling every conjugacy class of order-three elements of $\Lambda$ has a representative fixing a vertex of $\Delta$.
    \item Every point of $\mathbb{R}^2$ is within a distance of $t/3$ from a fixed point of an order-three rotation belonging to $\Lambda$.
\end{itemize}
\end{lemma}

\begin{proof} This is a consequence of the well known classification of groups of Euclidean isometries. First, any group $\Lambda$ satisfying the hypotheses is either a $(2,3,6)$- or $(3,3,3)$-rotation group. In the latter case, this means that there is a $\Lambda$-invariant tiling of $\mathbb{R}^2$ by equilateral triangles, with a fundamental domain for the $\Lambda$-action that is the union of two adjacent tiles, such that $\Lambda$ is generated by order-three rotations fixing the endpoints of the shared edge. Taking $a$ and $b$ to be distinct, counterclockwise such rotations, their product $ab$ fixes the third vertex of one of the two fundamental domain tiles and represents the final conjugacy class of order-three elements of $\Lambda$. 

The fourth fundamental domain vertex is fixed by $ba$, which is conjugate to $ab$. For every triangle of the $\Lambda$-invariant tiling, it follows that its three vertices are fixed by order-three rotations representing the three distinct conjugacy classes in $\Lambda$. Any $\lambda\in\Lambda$ carries the fixed point of a rotation $\rho\in\Lambda$ to the fixed point of $\lambda\rho\lambda^{-1}$. It follows that if $\lambda$ is a translation, its translation length must be at least the minimum distance between fixed points of rotations belonging to the same conjugacy class. If the tiles have side length $x$, this distance equals $x\sqrt{3}$, which is the distance between the fixed points of $ab$ and $ba$. Note that the lower bound is realized by the translation $bab$.

This implies the Lemma's conclusion (1) in the case that $\Lambda$ is a $(3,3,3)$-rotation group. Conclusion (2) follows by observing that in an equilateral triangle with sides of length $x$, the point at maximum distance from the set of vertices is the triangle's centroid, which is at distance $x/\sqrt{3}$ from each of them. Thus if $x = t/\sqrt{3}$, this distance is $t/3$, and the Lemma is proved in the case that $\Lambda$ is a $(3,3,3)$-rotation group.

In the case that $\Lambda$ is a $(2,3,6)$-rotation group, it has a unique index-two subgroup $\Lambda_0$, which is a $(3,3,3)$-rotation group. In this case, $\Lambda$ also leaves invariant the $\Lambda_0$-invariant tiling described above, and there is a choice of two-tile fundamental domain $F_0$ for $\Lambda_0$ such that $\Lambda$ is generated by an order-three rotation fixing the shared edge and an order-two rotation acting on it by reflection. Dividing $F_0$ in half by a perpendicular bisector of this edge thus yields a fundamental domain for $\Lambda$. Calling the order-three rotation $a$ and the order-two rotation $c$, we have that $b$ from before equals $cac$, and $ac$ is an order-six rotation with the same fixed point as $ab = (ac)^2$. It follows that the order-three fixed points of $\Lambda$ are those of $\Lambda_0$, and conclusions (1) and (2) follow from the previous case.
\end{proof}

\begin{lemma}\label{rotation center position}
 If a tiling $\mathcal{T}$ of $\mathbb{C}$ into squares and equilateral triangles of side length $1$ has an order $3$ rotational symmetry $r$, then the fixed point of $r$ is either the centroid of a triangle in $\mathcal{T}$ or a vertex  of a triangle in $\mathcal{T}$ that is not incident to any square of $\mathcal{T}$. In the latter case, six triangles of $\mathcal{T}$ meet at this vertex.
 \end{lemma}

 \begin{proof}
Let $p$ be the center of an order $3$ rotation $r$ of $\mathcal{T}$. Then $p$ is either a vertex of $\mathcal{T}$ or lies inside in the interior of an edge of $\mathcal{T}$ or lies inside the interior of a tile of $\mathcal{T}$. 

Now, $p$ cannot lie inside the interior of an edge of $\mathcal{T}$ otherwise $r$ would send that edge to another edge containing $p$ in the interior. But, $p$ cannot lie inside the interior of another edge of $\mathcal{T}$. 

If $p$ lies inside the interior of a tile then $r$ would have to send that tile to itself by acting as a $3$-cycle on its vertices. This would mean that the tile cannot be a square tile as otherwise it would contradict that its fixed point $p$ is in the interior. When the tile is triangular, $r$ would have to fix its centroid implying $p$ is the centroid.

If $p$ is a vertex of $\mathcal{T}$, $r$ cannot fix any tile that contains $p$ as a vertex. But, if $p$ is incident to a square tile, then $r$ would act as a $3$-cycle on the square tiles that contain $p$ as a vertex. This forces $p$ to be incident to only square tiles and as a consequence $r$ will have to fix a square tile incident to $p$, a contradiction. 

So, this shows that $p$ is either the centroid of a triangular tile or a vertex of $\mathcal{T}$ that is not incident to any square tile. In the latter case, $p$ is then necessarily the shared vertex of six triangular tiles of $\mathcal{T}$.
\end{proof}

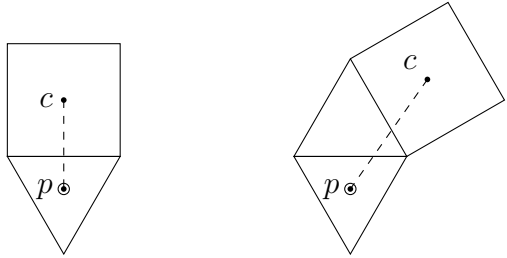
\begin{figure}[ht]
\begin{tikzpicture}

\begin{scope}[scale=0.75, xshift=-1in]

\draw (0,0) -- (2,0) -- (1,-1.732) -- cycle;
\draw (0,0) -- (2,0) -- (2,2) -- (0,2) -- cycle;

\fill (1,1) circle [radius=0.05];
\node [left] at (1,1) {$c$};

\fill (1,-0.577) circle [radius=0.05];
\draw (1,-0.577) circle [radius=0.1];
\node [left] at (1,-0.577) {$p$};
\draw [dashed] (1,1) -- (1,-0.577);

\end{scope}

\begin{scope}[scale=0.75, xshift=1in]

\draw (0,0) -- (2,0) -- (1,-1.732) -- cycle;
\draw (0,0) -- (2,0) -- (1,1.732) -- cycle;
\draw (2,0) -- (3.732,1) -- (2.732,2.732) -- (1,1.732);

\fill (2.366,1.366) circle [radius=0.05];
\node [above left] at (2.366,1.366) {$c$};

\fill (1,-0.577) circle [radius=0.05];
\draw (1,-0.577) circle [radius=0.1];
\node [left] at (1,-0.577) {$p$};
\draw [dashed] (2.366,1.366) -- (1,-0.577);

\end{scope}

\end{tikzpicture}
\caption{Possible nearest rotation fixed points to square centers.}
\label{nearest, dearest}
\end{figure}

\begin{lemma}\label{I'll kill you all!}
Suppose $\Lambda$ is a discrete, cocompact group of isometries of $\mathbb{R}^2$ containing rotations of order three that leaves invariant a tiling $\mathcal{T}$ of $\mathbb{C}$ into squares and equilateral triangles of side length $1$, and such that the minimal translation length of any infinite-order element of $\Lambda$ is at most $2+\sqrt{3}$. For a square tile of $\mathcal{T}$, referring to its center (ie.~the point equidistant from its vertices) as $c$, let $p$ be a nearest fixed point to $c$ of an order-three rotation belonging to $\Lambda$. The union of the tiles of $\mathcal{T}$ that intersect the line segment joining $p$ to $c$ is isometric to one of the two configurations of Figure \ref{nearest, dearest}.
\end{lemma}

\begin{proof}
By Lemma \ref{chow-dah! It's chow-dah!}, the distance from $c$ to $p$ is at most $(2+\sqrt{3})/3 < 1.245$. By Lemma \ref{rotation center position}, $p$ is either the centroid of a triangle of $\mathcal{T}$ or the shared vertex of six triangles of $\mathbb{T}$. For an equilateral triangle with sidelength $1$, the shortest distance from a vertex to the side opposite it is $\sqrt{3}/2$. It follows that if $p$ is the shared vertex of six equilateral triangles then there is an open disk of radius $\sqrt{3}/2$ centered at $p$ and contained in their union. Since there is an open disk of radius $1/2$ centered at $c$ and contained in the interior of the square tile containing it, the distance from $p$ to $c$ would be at least $1/2+\sqrt{3}/2 > 1.366$ in this case. But this is larger than $(2+\sqrt{3})/3$.

Therefore $p$ is the centroid of a triangle $T$ of $\mathcal{T}$. If the square tile containing $c$ is adjacent to $T$ then the union of tiles intersecting the line segment $[p,c]$ from $p$ to $c$ is as pictured on the left in Figure \ref{nearest, dearest}. So let us assume that this is not the case, and let $T_1$ be a tile of $\mathcal{T}$ adjacent to $T$ and intersecting $[p,c]$. We claim first that $T_1$ is not a square.

Supposing that $T_1$ is a square, we observe first that neither tile of $\mathcal{T}$ that intersects $T_1$ along an edge that shares a single vertex with $T_1\cap T$ is a square, since this is not compatible with the existence of an order-three rotation fixing $T$ and preserving $\mathcal{T}$. Therefore both such tiles are triangles, and the union of $T$, $T_1$, and these two is pictured on the left in Figure \ref{yacketty yak}.

\begin{figure}[ht]
\begin{tikzpicture}

\begin{scope}[scale=0.65, xshift=-1.5in]

\draw (0,0) -- (2,0) -- (1,-1.732) -- cycle;
\node [above left] at (1,-0.577) {\small $p$};
\node [below] at (1,-0.577) {$T$};
 
\draw (0,0) -- (2,0) -- (2,2) -- (0,2) -- cycle;
\node at (1,1) {$T_1$};

\draw (0,0) -- (0,2) -- (-1.732,1) -- cycle;
\draw (2,0) -- (2,2) -- (3.732,1) -- cycle;

\draw [dashed] (1,-0.577) -- (-3,1.732);
\draw [dashed] (1,-0.577) -- (5,1.732);
\fill (1,-0.577) circle [radius=0.05];
\draw (1,-0.577) circle [radius=0.1];

\end{scope}

\begin{scope}[scale=0.65, xshift=1.5in]

\draw (0,0) -- (2,0) -- (1,-1.732) -- cycle;
\node [above left] at (1,-0.577) {\small $p$};
\node [below] at (1,-0.577) {$T$};

\draw (0,0) -- (2,0) -- (1,1.732) -- cycle;
\node at (1,0.577) {$T_1$};

\draw (1,1.732) -- (3,1.732) -- (2,0);
\node at (2,1.155) {$T_2$};

\draw (3,1.732) -- (4,0) -- (2,0);
\fill (2,0) circle [radius=0.05];
\node [below] at (2,0) {\small $v$};

\draw [dashed] (1,-0.577) -- (5,1.732);
\fill (1,-0.577) circle [radius=0.05];
\draw (1,-0.577) circle [radius=0.1];

\end{scope}

\end{tikzpicture}
\caption{Two forbidden configurations.}
\label{yacketty yak}
\end{figure}
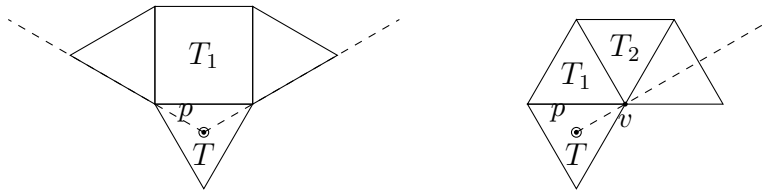

Since the line segment $[p,c]$ intersects the edge $T\cap T_1$, it exits the pictured union of triangles and squares within the cone defined by the dashed lines in the figure. However, the minimum distance from $p$ to the union of edges in the boundary of this region is $1+1/(2\sqrt{3}) > 1.288$, realized by the arc joining $p$ to the midpoint of the edge of $T_1$ opposite $T_1\cap T$. This is already larger than $(2+\sqrt{3})/3$, a contradiction that proves the claim.

Recall that the claim asserted that if $T_1$ does not contain $c$ then it is a triangle. We now assume this is the case and let $T_2\ne T_1$ be the tile of $\mathcal{T}$ that shares an edge with $T_1$ and intersects $[p,c]$. Let $v$ be the vertex of $T$ at which $T\cap T_1$ and $T_1\cap T_2$ meet. If $T_2$ were a triangle then by the rotation-invariance of $T$, $v$ would be contained in at least four---hence exactly six---triangles. Then $[p,c]$ would exit the union of the triangles containing $v$ in either $T_2$ or the unlabeled triangle on the right in Figure \ref{yacketty yak}. The distance from $p$ to the closest point of the edge of either of these triangles is $\sqrt{3}/2+1/(2\sqrt{3}) > 1.154$. The length of $[p,c]$ is at least $1/2$ larger than this, since this is a lower bound on the length of its intersection with the square tile containing $c$. Since this is larger than $(2+\sqrt{3})/3$, $T_2$ cannot be a triangle.

Therefore in the case that $T_1$ does not contain $c$, it is a triangle and $T_2\ne T$---the tile of $\mathcal{T}$ that shares an edge with $T_1$ and intersects $[p,c]$---is a square. By arguing as in the case of the paragraph above we can show that if $c$ were not the center of $T_2$ then $[p,c]$ would be too long; so $c$ is the center of $T_2$ and, up to isometry, the configuration $T\cup T_1\cup T_2$ is as pictured on the right in Figure \ref{nearest, dearest}.\end{proof}

\begin{figure}[ht]
\begin{tikzpicture}

\begin{scope}[scale=0.5, xshift=-2in]

\fill [opacity=0.1] (0,0) -- (2,0) -- (1,-1.732) -- cycle;
\fill [opacity=0.1] (0,0) -- (2,0) -- (2,2) -- (0,2) -- cycle;
\fill [opacity=0.1] (2,0) -- (3.732,-1) -- (2.732,-2.732) -- (1,-1.732);
\fill [opacity=0.1] (0,0) -- (-1.732,-1) -- (-0.732,-2.732) -- (1,-1.732);

\draw (0,0) -- (2,0) -- (1,-1.732) -- cycle;
\draw (0,0) -- (2,0) -- (2,2) -- (0,2) -- cycle;
\draw (2,2) -- (3.732,1) -- (2,0) -- (3.732,-1) -- (3.732,1);
\draw (0,2) -- (-1.732,1) -- (0,0) -- (-1.732,-1) -- (-1.732,1);
\draw (3.732,-1) -- (2.732,-2.732) -- (1,-1.732) -- (1,-3.732) -- (2.732,-2.732);
\draw (-1.732,-1) -- (-0.732,-2.732) -- (1,-1.732);
\draw (-0.732,-2.732) -- (1,-3.732);

\fill (1,-0.577) circle [radius=0.05];
\draw (1,-0.577) circle [radius=0.1];
\node [left] at (1,-0.577) {\small $p$};

\end{scope}

\begin{scope}[scale=0.5, xshift=2in]

\fill [opacity=0.1] (0,0) -- (2,0) -- (1,-1.732) -- cycle;
\fill [opacity=0.1] (0,0) -- (2,0) -- (1,1.732) -- cycle;
\fill [opacity=0.1] (2,0) -- (3.732,1) -- (2.732,2.732) -- (1,1.732);
\fill [opacity=0.1] (2,0) -- (3,-1.732) -- (3,-3.732) -- (1,-3.732) -- (1,-1.732);
\fill [opacity=0.1] (1,-1.732) -- (-1,-1.732) -- (-2.732,-0.732) -- (-1.732,1) -- (0,0);

\draw (0,0) -- (2,0) -- (1,-1.732) -- cycle;
\draw (0,0) -- (2,0) -- (1,1.732) -- cycle;
\draw (2,0) -- (3.732,1) -- (2.732,2.732) -- (1,1.732);
\draw (2,0) -- (3,-1.732) -- (-1,-1.732) -- (0,0);
\draw (3.732,1) -- (4.732,-0.732) -- (3,-1.732) -- (4.732,-2.732) -- (4.732,-0.732);
\draw (4.732,-2.732) -- (3,-3.732) -- (3,-1.732);
\draw (3,-3.732) -- (-1,-3.732) -- (-1,-1.732) -- (-2.732,-2.732) -- (-1,-3.732);
\draw (1,-1.732) -- (1,-3.732);
\draw (-2.732,-2.732) -- (-2.732,-0.732) -- (-1,-1.732);
\draw (-2.732,-0.732) -- (-0.732,2.732) -- (1,1.732) -- (1,3.732) -- (2.732,2.732);
\draw (-0.732,2.732) -- (1,3.732);
\draw (0,0) -- (-1.732,1);

\fill (1,-0.577) circle [radius=0.05];
\draw (1,-0.577) circle [radius=0.1];
\node [left] at (1,-0.577) {\small $p$};

\end{scope}

\end{tikzpicture}
\caption{Rotation-Invariant Square Carriers.}
\label{blobby blobs}
\end{figure}
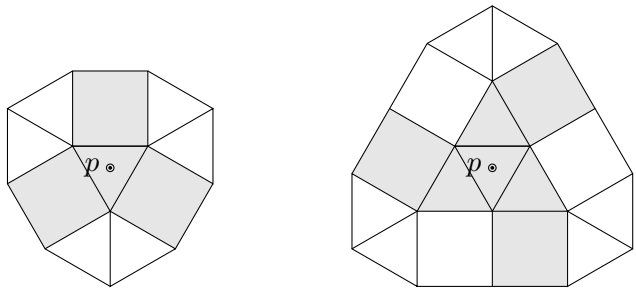

\begin{corollary}\label{Especially those of you on the jury.}
    Suppose $\Lambda$ is a discrete, cocompact group of isometries of $\mathbb{R}^2$ containing rotations of order three that leaves invariant a tiling $\mathcal{T}$ of $\mathbb{C}$ into squares and equilateral triangles of side length $1$, and such that the minimal translation length of any infinite-order element of $\Lambda$ is at most $2+\sqrt{3}$. Every square tile of $\mathcal{T}$ is contained in a region isometric to one of the two pictured in Figure \ref{blobby blobs}, with a fixed point of a rotation of $\Lambda$ at its center.
\end{corollary}

\begin{proof} Fix a square tile $S$ of $\mathcal{T}$, and call its center $c$. By Lemma \ref{I'll kill you all!}, up to isometry $S$ is the square tile in one of the two configurations pictured in Figure \ref{nearest, dearest}, where $p$ is a nearest fixed point to $c$ of an order-three rotation $\lambda\in\Lambda$. The union of this configuration with its images under $\lambda$ and $\lambda^{-1}$ is one of the two shaded regions in Figure \ref{blobby blobs}. The remainder of each neighborhood in the Figure is then determined by the shaded region and the geometry of the two tile types in $\mathcal{T}$.\end{proof}

\begin{notation} For the rest of this section, we will refer to the neighborhoods pictured in Figure \ref{blobby blobs} as \textit{Rotation-Invariant Square Carriers}, or \textit{RISC}s for short.
\end{notation}

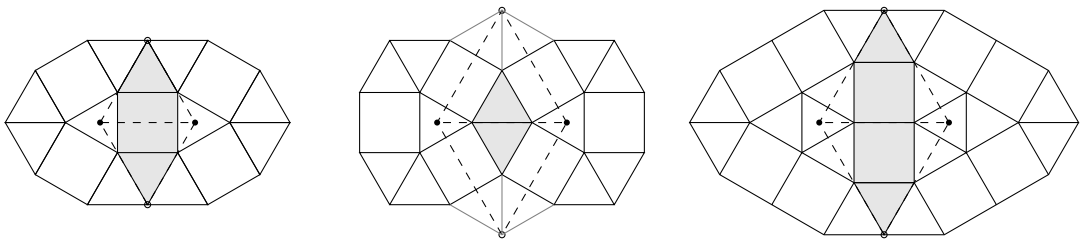
\begin{figure}[ht]
\begin{tikzpicture}

\begin{scope}[scale=0.4,xshift=-1.5in]

\draw (0,1.732) -- (1.732,2.732) -- (2.732,1) -- (1,0) -- (2.732,-1) -- (1.732,-2.732) -- (0,-1.732) -- (-1.732,-2.732) -- (-2.732,-1) -- (-1,0) -- (-2.732,1) -- (-1.732,2.732) -- cycle;
\draw (1.732,2.732) -- (3.732,2.732) -- (2.732,1) -- (4.732,1) -- (3.732,2.732);
\draw (1.732,-2.732) -- (3.732,-2.732) -- (2.732,-1) -- (4.732,-1) -- (3.732,-2.732);
\draw (-1.732,2.732) -- (-3.732,2.732) -- (-2.732,1) -- (-4.732,1) -- (-3.732,2.732);
\draw (-1.732,-2.732) -- (-3.732,-2.732) -- (-2.732,-1) -- (-4.732,-1) -- (-3.732,-2.732);
\draw (2.732,1) -- (2.732,-1);
\draw (4.732,1) -- (4.732,-1);
\draw (-2.732,1) -- (-2.732,-1);
\draw (-4.732,1) -- (-4.732,-1);

\fill [opacity=0.1] (-1,0) -- (0,1.732) -- (1,0) -- (0,-1.732) -- cycle;
\draw (-1,0) -- (0,1.732) -- (1,0) -- (0,-1.732) -- (-1,0) -- (1,0);

\fill (2.155,0) circle [radius=0.1];
\fill (-2.155,0) circle [radius=0.1];
\draw (0,3.732) circle [radius=0.1];
\draw (0,-3.732) circle [radius=0.1];
\draw [dashed] (2.155,0) -- (0,3.732) -- (-2.155,0) -- (0,-3.732) -- (2.155,0) -- (-2.155,0);

\draw [color=gray] (-1.732,2.732) -- (0,3.732) -- (1.732,2.732);
\draw [color=gray] (0,3.732) -- (0,1.732);
\draw [color=gray] (-1.732,-2.732) -- (0,-3.732) -- (1.732,-2.732);
\draw [color=gray] (0,-3.732) -- (0,-1.732);

\end{scope}

\begin{scope}[scale=0.8, xshift=-4in]
\fill[opacity=0.1] (1.866,0.5)--(2.366,1.366)--(2.866,0.5)--(2.866,-0.5)--(2.366,-1.366)--(1.866,-0.5)--cycle;
\draw (0,0) -- (1,0) -- (0.5,0.866) -- cycle;
\draw [xshift=1.866cm, yshift=0.5cm] (0,0) -- (1,0) -- (0.5,0.866) -- cycle;
\draw [xshift=3.732cm] (0,0) -- (1,0) -- (0.5,0.866) -- cycle;
\draw [xshift=1.366cm, yshift=-1.366cm] (0,0) -- (1,0) -- (0.5,0.866) -- cycle;
\draw [xshift=2.366cm, yshift=-1.366cm] (0,0) -- (1,0) -- (0.5,0.866) -- cycle;

\draw (0,0) -- (1,0) -- (0.5,-0.866) -- cycle;
\draw [xshift=3.732cm] (0,0) -- (1,0) -- (0.5,-0.866) -- cycle;
\draw [xshift=1.366cm, yshift=1.366cm] (0,0) -- (1,0) -- (0.5,-0.866) -- cycle;
\draw [xshift=2.366cm, yshift=1.366cm] (0,0) -- (1,0) -- (0.5,-0.866) -- cycle;

\draw (1.866,0.5) -- (1.866,-0.5) -- (2.866,-0.5) -- (2.866,0.5);

\draw (1,0) -- (1.866,0.5) -- (1.366,1.366) -- (0.5,0.866) -- cycle;
\draw [xshift=2.366cm, yshift=-1.366cm] (1,0) -- (1.866,0.5) -- (1.366,1.366) -- (0.5,0.866) -- cycle;

\draw (1,0) -- (1.866,-0.5) -- (1.366,-1.366) -- (0.5,-0.866) -- cycle;
\draw [xshift=2.366cm, yshift=1.366cm] (1,0) -- (1.866,-0.5) -- (1.366,-1.366) -- (0.5,-0.866) -- cycle;

\draw [dashed] (1.577,0)--(2.366,1.366)--(3.155,0)--(2.366,-1.366)--(1.577,0)--(3.155,0);
\fill (1.577,0) circle [radius=0.05];
\fill (3.155,0) circle [radius=0.05];
\draw (2.366,1.366) circle [radius=0.05];
\draw (2.366,-1.366) circle [radius=0.05];

\end{scope}

\begin{scope}[scale=0.4, xshift=3.5in]

\fill [opacity=0.1] (-1,2) -- (0,3.732) -- (1,2) -- (1,-2) -- (0,-3.732) -- (-1,-2) -- cycle;
\draw (-1,2) -- (1,2) -- (1,-2) -- (-1,-2) -- cycle;

\draw (-1,0) -- (1,0) -- (2.732,1) -- (1,2) -- (0,3.732) -- (2,3.732) -- (3.732,2.732) -- (2.732,1) -- (2.732,-1) -- (3.732,-2.732) -- (2,-3.732) -- (0,-3.732) -- (1,-2) -- (2.732,-1) -- (1,0);
\draw (-1,0) -- (-2.732,1) -- (-1,2) -- (0,3.732) -- (-2,3.732) -- (-3.732,2.732) -- (-2.732,1) -- (-2.732,-1) -- (-3.732,-2.732) -- (-2,-3.732) -- (0,-3.732) -- (-1,-2) -- (-2.732,-1) -- (-1,0);
\draw (-2,3.732) -- (-1,2) -- (1,2) -- (2,3.732);
\draw (-2,-3.732) -- (-1,-2) -- (1,-2) -- (2,-3.732);
\draw (2.732,1) -- (4.464,0) -- (5.464,1.732) -- (3.732,2.732);
\draw (2.732,-1) -- (4.464,0) -- (5.464,-1.732) -- (3.732,-2.732);
\draw (-2.732,1) -- (-4.464,0) -- (-5.464,1.732) -- (-3.732,2.732);
\draw (-2.732,-1) -- (-4.464,0) -- (-5.464,-1.732) -- (-3.732,-2.732);
\draw (4.464,0) -- (6.464,0) -- (5.464,1.732);
\draw (6.464,0) -- (5.464,-1.732);
\draw (-4.464,0) -- (-6.464,0) -- (-5.464,1.732);
\draw (-6.464,0) -- (-5.464,-1.732);

\fill (2.155,0) circle [radius=0.1];
\fill (-2.155,0) circle [radius=0.1];
\draw (0,3.732) circle [radius=0.1];
\draw (0,-3.732) circle [radius=0.1];
\draw [dashed] (2.155,0) -- (0,3.732) -- (-2.155,0) -- (0,-3.732) -- (2.155,0) -- (-2.155,0);

\end{scope}

\end{tikzpicture}

    \caption{Unions of RISC neighborhoods with centers at distance $\le 2+\sqrt{3}$; overlap shaded. Additional gray triangles in the middle figure belong to any tiling by equilateral triangles and squares which contains that configuration. The group generated by order-three rotations fixing the two RISC centers extends the pictured tilings to tilings of $\mathbb{R}^2$ that we respectively call $\mathcal{Q}$ (also pictured in \Cref{tiddly}), $\totaleclipse$ (also in \Cref{fig_total_eclipse}) and $\mathcal{R}$ (also in \Cref{fig_ace}).}
    \label{dumb bell}
\end{figure}

\begin{lemma}\label{And then there was shmoo}
Suppose $\Lambda$ is a discrete, cocompact group of isometries of $\mathbb{R}^2$ containing rotations of order three that leaves invariant a tiling $\mathcal{T}$ of $\mathbb{C}$ into squares and equilateral triangles of side length $1$---with each shape represented. If $\mathcal{T}$ contains a pair of distinct RISC neighborhoods with rotation centers at a distance less than $\ell= 2+1/\sqrt{3}$, then the union of these two neighborhoods forms one of the three configurations of Figure \ref{dumb bell}.
\end{lemma}

\begin{remark}\label{embark} In each case of the Figure, the union of the two dashed equilateral triangles forms a fundamental domain for the subgroup of $\Lambda$ generated by the order-three rotations at the RISC rotation centers. Hence $\mathcal{T}$ is determined in each case by its intersection with this fundamental domain.
\end{remark}

\begin{proof}[Proof of Lemma \ref{And then there was shmoo}]
A Euclidean geometry computation shows that the left-hand RISC neighborhood of Figure \ref{blobby blobs} contains an open disk of radius $1+1/(2\sqrt{3}) = \ell/2$ centered at its rotation center $p$, and the right-hand neighborhood contains one of radius $1+1/\sqrt{3}>\ell/2$. Therefore, RISC neighborhoods with rotation centers at a distance less than $\ell$ must overlap.

Two RISCs that overlap must do so in a union of tiles of $\mathcal{T}$, and their combinatorics constrain the possibilities. In particular, the left- and right-hand RISCs of Figure \ref{blobby blobs} can only overlap each other in a union of adjacent triangles, and this puts their centers at distance $1/\sqrt{3} + 1 + \sqrt{3}/2 + 1/(2\sqrt{3}) > \ell$. Therefore any two RISC neighborhoods with rotation centers at a distance less than $\ell$ must consist of two copies of the same type.

For overlapping copies of a single RISC type, their region of overlap either contains a square, or it is a union of two adjacent triangles. For the RISC on the left in Figure \ref{blobby blobs}, the former region of overlap puts the rotation centers a distance of $1+1/\sqrt{3}$ apart, and the union of the two neighborhoods is pictured on the left in Figure \ref{dumb bell}. The latter puts the centers $1+2/\sqrt{3}$ apart and is pictured in the middle of the Figure. For the RISC on the right in Figure \ref{blobby blobs}, a region of overlap containing squares yields a distance of $1+2/\sqrt{3}$. This configuration is pictured on the right in Figure \ref{dumb bell}. On the other hand, overlapping in two adjacent triangles puts the distance at $1+\sqrt{3}+1/\sqrt{3}>\ell$, thus not satisfying the Lemma's hypotheses.
\end{proof}

\begin{prop}\label{And then there were two}
Suppose $\Lambda$ is a discrete, cocompact group of isometries of $\mathbb{R}^2$ containing rotations of order three that leaves invariant a tiling $\mathcal{T}$ of $\mathbb{C}$ into squares and equilateral triangles of side length $1$---with each shape represented---and such that the minimal translation length $\tau$ among infinite-order elements of $\Lambda$ satisfies $\tau\le 2+\sqrt{3}$. Then $\mathcal{T}$ is one of the tilings generated from those in Figure \ref{dumb bell} as in Remark \ref{embark}.
\end{prop}

\begin{remark}
   From left to right, the tilings of \Cref{And then there were two} are also pictured on the left in Figure \ref{tiddly}; as \totaleclipse\ from \Cref{fig_total_eclipse}; and as the tiling \aceofbase\ pictured in \Cref{fig_ace}.
\end{remark}

\begin{proof}
By Lemma \ref{chow-dah! It's chow-dah!}, there is a tiling of $\mathbb{R}^2$ by equilateral triangles (only---this is an overlay of $\mathcal{T}$) with sides of length $\tau/\sqrt{3}$ and vertex set equal to the set of fixed points of order-three rotations of $\Lambda$. Fix a tile $\Delta$ of this tiling by equilateral triangles. Since $\mathcal{T}$ contains a square, by Corollary \ref{Especially those of you on the jury.} at least one vertex of the tiling by equilateral triangles has a neighborhood in $\mathcal{T}$ isometric to one of those pictured in Figure \ref{blobby blobs}, a RISC; and therefore, by Lemma \ref{chow-dah! It's chow-dah!}(1), at least one vertex of $\Delta$ has such a neighborhood.

If two vertices of $\Delta$ have RISC neighborhoods then since $\tau/\sqrt{3}$ is less than $\ell$ as defined in \Cref{And then there was shmoo}, that result plus Remark \ref{embark} imply the conclusion of this one. For the rest of this proof, we therefore consider the case that exactly one vertex of $\Delta$, call it $p$, is a RISC rotation center. If $e$ is the edge of $\Delta$ opposite $p$, $\bar\Delta\ne \Delta$ is the equilateral triangle sharing $e$ with $\Delta$, and $\bar{p}$ is the vertex of $\bar\Delta$ opposite $e$, then $\bar{p}$ is at distance $\tau$ from $p$ and is its image under a translation belonging to $\Lambda$. It is thus also a RISC rotation center, and we may assume that $\tau\geq \ell$, for $\ell$ as in \Cref{And then there was shmoo}, since otherwise that result plus Remark \ref{embark} again classify all possibilities.

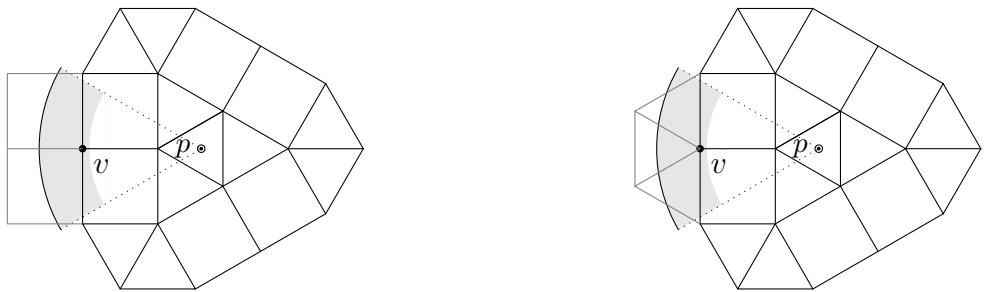
\begin{figure}[ht]
\begin{tikzpicture}

\begin{scope}[scale=0.5, rotate=30, xshift=-3in]

\draw [dotted] (-3.309,-0.577) -- (1,-0.577) -- (-1.155,3.155);
\fill [opacity=0.1] (-3.309,-0.577) arc (180:120:4.309);
\fill [opacity=0.1] (-3.309,-0.577)--(-1.976,-0.577)--(-0.488,2)--(-1.155,3.155);
\fill [color=white] (-1.976,-0.577) arc (180:120:2.976);
\draw (-3.309,-0.577) arc (180:120:4.309);

\draw (0,0) -- (2,0) -- (1,-1.732) -- cycle;
\draw (0,0) -- (2,0) -- (1,1.732) -- cycle;
\draw (2,0) -- (3.732,1) -- (2.732,2.732) -- (1,1.732);
\draw (2,0) -- (3,-1.732) -- (-1,-1.732) -- (0,0);
\draw (3.732,1) -- (4.732,-0.732) -- (3,-1.732) -- (4.732,-2.732) -- (4.732,-0.732);
\draw (4.732,-2.732) -- (3,-3.732) -- (3,-1.732);
\draw (3,-3.732) -- (-1,-3.732) -- (-1,-1.732) -- (-2.732,-2.732) -- (-1,-3.732);
\draw (1,-1.732) -- (1,-3.732);
\draw (-2.732,-2.732) -- (-2.732,-0.732) -- (-1,-1.732);
\draw (-2.732,-0.732) -- (-0.732,2.732) -- (1,1.732) -- (1,3.732) -- (2.732,2.732);
\draw (-0.732,2.732) -- (1,3.732);
\draw (0,0) -- (-1.732,1);

\fill (1,-0.577) circle [radius=0.05];
\draw (1,-0.577) circle [radius=0.1];
\node [left] at (1,-0.577) {\small $p$};

\fill (-1.732,1) circle [radius=0.1];
\node [below right] at (-1.732,1) {$v$};
\draw [gray] (-1.732,1) -- (-3.464,2) -- (-4.464,0.268) -- (-2.732,-0.732);
\draw [gray] (-3.464,2) -- (-2.464,3.732) -- (-0.732,2.732);

\end{scope}

\begin{scope}[scale=0.5, xshift=3in, rotate=30, yshift=-1.732in]

\draw [dotted] (-3.309,-0.577) -- (1,-0.577) -- (-1.155,3.155);
\fill [opacity=0.1] (-3.309,-0.577) arc (180:120:4.309);
\fill [opacity=0.1] (-3.309,-0.577)--(-1.976,-0.577)--(-0.488,2)--(-1.155,3.155);
\fill [color=white] (-1.976,-0.577) arc (180:120:2.976);
\draw (-3.309,-0.577) arc (180:120:4.309);

\draw (0,0) -- (2,0) -- (1,-1.732) -- cycle;
\draw (0,0) -- (2,0) -- (1,1.732) -- cycle;
\draw (2,0) -- (3.732,1) -- (2.732,2.732) -- (1,1.732);
\draw (2,0) -- (3,-1.732) -- (-1,-1.732) -- (0,0);
\draw (3.732,1) -- (4.732,-0.732) -- (3,-1.732) -- (4.732,-2.732) -- (4.732,-0.732);
\draw (4.732,-2.732) -- (3,-3.732) -- (3,-1.732);
\draw (3,-3.732) -- (-1,-3.732) -- (-1,-1.732) -- (-2.732,-2.732) -- (-1,-3.732);
\draw (1,-1.732) -- (1,-3.732);
\draw (-2.732,-2.732) -- (-2.732,-0.732) -- (-1,-1.732);
\draw (-2.732,-0.732) -- (-0.732,2.732) -- (1,1.732) -- (1,3.732) -- (2.732,2.732);
\draw (-0.732,2.732) -- (1,3.732);
\draw (0,0) -- (-1.732,1);

\fill (1,-0.577) circle [radius=0.05];
\draw (1,-0.577) circle [radius=0.1];
\node [left] at (1,-0.577) {\small $p$};

\fill (-1.732,1) circle [radius=0.1];
\node [below right] at (-1.732,1) {$v$};
\draw [gray] (-2.732,-0.732) -- (-3.732,1) -- (-1.732,1) -- (-2.732,2.732) -- (-3.732,1);
\draw [gray] (-2.732,2.732) -- (-0.732,2.732);

\end{scope}
\end{tikzpicture}
\caption{Two possible neighborhoods of a vertex $v$ on $\partial(\mbox{RISC})$.}
\label{location location location}
\end{figure}

We will show that in fact it is not possible for a unique vertex of $\Delta$ to be a RISC rotation center. We claim first:

\begin{claim}\label{clam}
    For $\ell \le \tau \le 2+\sqrt{3}$, if a unique vertex of $\Delta$ is a RISC rotation center, it does not belong to the RISC pictured on the right in Figure \ref{blobby blobs}.
\end{claim}

 \begin{proof}[Proof of Claim] 
 Consider possible locations of order-three rotation centers of $\Lambda$, relative to a ray $\rho$ from the center $p$ of the RISC that splits a pair of its adjacent squares. Let $v$ be the vertex of intersection between $\rho$ and the RISC's boundary. The only two possible neighborhoods of $v$ in $\mathcal{T}$ are pictured in Figure \ref{location location location}---with either two squares (on the left) or three triangles (right) drawn in gray outside the RISC---along with an arc $\alpha$ of angular measure $60$ degrees centered at the intersection between $\rho$ and a circle of radius $1+2/\sqrt{3}$ centered at $p$. The shaded regions in the Figure consist of points with these angular values and between circles of radii the upper and lower bounds on $\tau/\sqrt{3}$:  
 \[ \frac{1}{3} + \frac{2}{\sqrt{3}} \le \frac{\tau}{\sqrt{3}} \le 1 + \frac{2}{\sqrt{3}}. \]
 At least one order-three rotation center of $\Lambda$ must lie in the shaded region, since there are six such centers total, evenly spaced around the circle of radius $\tau/\sqrt{3}$ centered at $p$.

The neighborhood of $v$ pictured on the left in Figure \ref{location location location} is not compatible with the existence of an order-three rotation center in the shaded region, since by Lemma \ref{rotation center position}, such a rotation center must occur at the centroid or a vertex of a triangle of $\mathcal{T}$. The neighborhood pictured on the right does encompass the case pictured on the right in Figure \ref{dumb bell}, in which the rotation center lies on the ray $\rho$ at the centroid of the middle triangle, at a distance attaining the upper bound $1+2/\sqrt{3}$ for $\tau/\sqrt{3}$. In this case, applying the order-three rotation around this point and its inverse duplicates the previous Figure. In particular, in this case there are distinct copies of this RISC centered at adjacent vertices; hence at adjacent vertices of $\Delta$.

In the right-hand case of Figure \ref{location location location}, neither of the other two triangles outside the RISC can be preserved by an order-three rotation, since each intersects both a triangle and a square. And the parts of the shaded region that lie outside its intersection with the pictured neighborhood of $v$ are contained in a union of squares and/or triangles that intersect both a square and a triangle. Thus no other location in the shaded region can be the fixed point of an order-three rotation that preserves $\mathcal{T}$, proving the claim.
\end{proof}

Given the claim, we are left to consider the RISC pictured on the left in Figure \ref{blobby blobs}, in the case that $\ell\le \tau\le 2+\sqrt{3}$ and its center $p$ is at the unique vertex of $\Delta$ which is a RISC rotation center. Hence as observed above, the nearest other RISC rotation center is a translate of this one at a distance $\tau$ away. A Euclidean geometry computation shows that this RISC type is entirely contained in a closed disk of radius $1+1/\sqrt{3}$ centered at $p$. If this is less than half of $\tau$, then since every square tile of $\mathcal{T}$ is contained in a RISC (by Lemma \ref{Especially those of you on the jury.}), this RISC would intersect only triangles of $\mathcal{T}$. But this is not possible, since its boundary vertices which are contained in one square must also be contained in another. Therefore in this case we must have $2+2/\sqrt{3}\ge\tau\ge\ell\doteq 2+1/\sqrt{3}$.

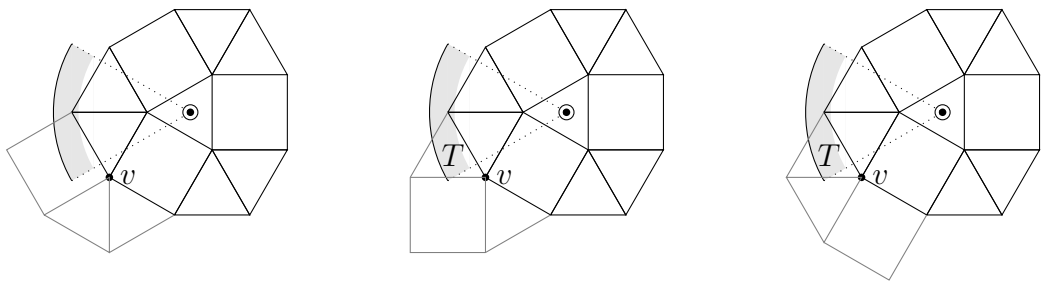
\begin{figure}[ht]
\begin{tikzpicture}

\begin{scope}[xshift=-5cm]

\draw [dotted] (0,-0.911) -- (1.577,0) -- (0,0.911);
\draw (0,-0.911) arc (210:150:1.821);
\fill [opacity=0.1] (0,-0.911) arc (210:150:1.821);
\fill [opacity=0.1] (0,-0.911) -- (0.289,-0.744) -- (0.289,0.744) -- (0,0.911) -- cycle;
\fill [color=white] (0.289,-0.744) arc (210:150:1.488);

\draw (0,0) -- (1,0) -- (0.5,0.866) -- cycle;
\draw [xshift=1.866cm, yshift=0.5cm] (0,0) -- (1,0) -- (0.5,0.866) -- cycle;
\draw [xshift=1.366cm, yshift=-1.366cm] (0,0) -- (1,0) -- (0.5,0.866) -- cycle;

\draw (0,0) -- (1,0) -- (0.5,-0.866) -- cycle;
\draw [xshift=1.366cm, yshift=1.366cm] (0,0) -- (1,0) -- (0.5,-0.866) -- cycle;
\draw [xshift=1.866cm, yshift=-0.5cm] (0,0) -- (1,0) -- (0.5,-0.866) -- cycle;

\draw (1.866,0.5) -- (1.866,-0.5) -- (2.866,-0.5) -- (2.866,0.5);

\draw (1,0) -- (1.866,0.5) -- (1.366,1.366) -- (0.5,0.866) -- cycle;

\draw (1,0) -- (1.866,-0.5) -- (1.366,-1.366) -- (0.5,-0.866) -- cycle;

\fill (1.577,0) circle [radius=0.05];
\draw (1.577,0) circle [radius=0.1];

\fill (0.5,-0.866) circle [radius=0.05];
\node [right] at (0.5,-0.866) {$v$};

\draw [color=gray] (0,0) -- (-0.866,-0.5) -- (-0.366,-1.366) -- (0.5,-0.866) -- (0.5,-1.866) -- (-0.366,-1.366);
\draw [color=gray] (0.5,-1.866) -- (1.366,-1.366);

\end{scope}

\begin{scope}

\draw [dotted] (0,-0.911) -- (1.577,0) -- (0,0.911);
\draw (0,-0.911) arc (210:150:1.821);
\fill [opacity=0.1] (0,-0.911) arc (210:150:1.821);
\fill [opacity=0.1] (0,-0.911) -- (0.289,-0.744) -- (0.289,0.744) -- (0,0.911) -- cycle;
\fill [color=white] (0.289,-0.744) arc (210:150:1.488);

\draw (0,0) -- (1,0) -- (0.5,0.866) -- cycle;
\draw [xshift=1.866cm, yshift=0.5cm] (0,0) -- (1,0) -- (0.5,0.866) -- cycle;
\draw [xshift=1.366cm, yshift=-1.366cm] (0,0) -- (1,0) -- (0.5,0.866) -- cycle;

\draw (0,0) -- (1,0) -- (0.5,-0.866) -- cycle;
\draw [xshift=1.366cm, yshift=1.366cm] (0,0) -- (1,0) -- (0.5,-0.866) -- cycle;
\draw [xshift=1.866cm, yshift=-0.5cm] (0,0) -- (1,0) -- (0.5,-0.866) -- cycle;

\draw (1.866,0.5) -- (1.866,-0.5) -- (2.866,-0.5) -- (2.866,0.5);

\draw (1,0) -- (1.866,0.5) -- (1.366,1.366) -- (0.5,0.866) -- cycle;

\draw (1,0) -- (1.866,-0.5) -- (1.366,-1.366) -- (0.5,-0.866) -- cycle;

\fill (1.577,0) circle [radius=0.05];
\draw (1.577,0) circle [radius=0.1];

\fill (0.5,-0.866) circle [radius=0.05];
\node [right] at (0.5,-0.866) {$v$};

\node at (0.07,-0.57) {$T$};

\draw [color=gray] (0,0) -- (-0.5,-0.866) -- (-0.5,-1.866) -- (0.5,-1.866) -- (1.366,-1.366);
\draw [color=gray] (-0.5,-0.866) -- (0.5,-0.866) -- (0.5,-1.866);

\end{scope}

\begin{scope}[xshift=5cm]

\draw [dotted] (0,-0.911) -- (1.577,0) -- (0,0.911);
\draw (0,-0.911) arc (210:150:1.821);
\fill [opacity=0.1] (0,-0.911) arc (210:150:1.821);
\fill [opacity=0.1] (0,-0.911) -- (0.289,-0.744) -- (0.289,0.744) -- (0,0.911) -- cycle;
\fill [color=white] (0.289,-0.744) arc (210:150:1.488);

\draw (0,0) -- (1,0) -- (0.5,0.866) -- cycle;
\draw [xshift=1.866cm, yshift=0.5cm] (0,0) -- (1,0) -- (0.5,0.866) -- cycle;
\draw [xshift=1.366cm, yshift=-1.366cm] (0,0) -- (1,0) -- (0.5,0.866) -- cycle;

\draw (0,0) -- (1,0) -- (0.5,-0.866) -- cycle;
\draw [xshift=1.366cm, yshift=1.366cm] (0,0) -- (1,0) -- (0.5,-0.866) -- cycle;
\draw [xshift=1.866cm, yshift=-0.5cm] (0,0) -- (1,0) -- (0.5,-0.866) -- cycle;

\draw (1.866,0.5) -- (1.866,-0.5) -- (2.866,-0.5) -- (2.866,0.5);

\draw (1,0) -- (1.866,0.5) -- (1.366,1.366) -- (0.5,0.866) -- cycle;

\draw (1,0) -- (1.866,-0.5) -- (1.366,-1.366) -- (0.5,-0.866) -- cycle;

\fill (1.577,0) circle [radius=0.05];
\draw (1.577,0) circle [radius=0.1];

\fill (0.5,-0.866) circle [radius=0.05];
\node [right] at (0.5,-0.866) {$v$};

\node at (0.07,-0.57) {$T$};

\draw [color=gray] (0,0) -- (-0.5,-0.866) -- (0.5,-0.866) -- (0,-1.732) -- (-0.5,-0.866);
\draw [color=gray] (0,-1.732) -- (0.866,-2.232) -- (1.366,-1.366);

\end{scope}

\end{tikzpicture}
\caption{Possible rotation center locations.}
\label{lil small guy}
\end{figure}

In parallel to the proof of Claim \ref{clam}, we now consider possible locations of rotation centers at a distance of $\tau/\sqrt{3}$ from $p$. At least one must lie on a sixty-degree arc centered on a ray $\rho$ from $p$ that splits a pair of triangles that share an edge. Given the constraints on $\tau$, this means that such a rotation center would lie in the shaded region(s) in Figure \ref{lil small guy}. This Figure's three cases depict the three distinct possibilities for the combinatorics of a tiled neighborhood of the vertex labeled $v$ on the RISC's boundary.

Since the RISC is invariant under reflection through the line containing $\rho$, after applying such a reflection if necessary, we may assume that there is a rotation center on the side of $\rho$ containing $v$. We note first that such a rotation center cannot lie outside the pictured tiled neighborhood of $v$ in any of the three cases, since in all cases the centroids of any other possible (non-pictured) triangular tiles are not close enough to $p$ to lie in the shaded neighborhood. But the tiled neighborhood of $v$ also cannot contain a rotation center in any of the three cases: in the leftmost, its intersection with the shaded neighborhood is contained in a square; in the other two its intersection with the shaded neighborhood contains the centroid of a unique triangle (``$T$'' in Figure \ref{lil small guy}), but the neighborhood's combinatorics prohibit this centroid from being a rotation center for $\mathcal{T}$. In the center case we can see this from the fact that $T$ shares one edge with a triangle and another with the square. In the right-hand case, while $T$ shares edges only with triangles, one of these triangles abuts a square along its edge clockwise from its intersection with $T$, but the other one abuts a triangle along its corresponding edge.

Between this case and Claim \ref{clam}, we have now showed that $\Delta$ cannot have a unique vertex that is a RISC rotation center. Therefore as discussed at the beginning of the proof, the classification of possibilities follows from \Cref{And then there was shmoo} and Remark \ref{embark}.
\end{proof}

\section{Eliminating the remaining peripheral tiling}\label{eclipse}

The main result of this section, Proposition \ref{slop} below, implies that the tiling \totaleclipse\ pictured in \Cref{fig_total_eclipse} is not the peripheral tiling of a mixed-platonic knot complement with hidden symmetries. This holds even though the full orientation-preserving symmetry group of \totaleclipse\ is a $(2,3,6)$-rotation group, as will be explicitly exhibited in the proof of Proposition \ref{slop}. It further implies the main theorem \ref{main_thm}, since \totaleclipse\ is the final possible peripheral tiling of such knots that has remained unaddressed to this point.

\newcommand\SlopPropStmt{There is no orientable one-cusped, mixed-platonic orbifold $O = \mathbb{H}^3/\Gamma$ with peripheral tiling $\totaleclipse$, such that the peripheral subgroup of $\Gamma$ is the full orientation-preserving symmetry group $\text{Sym}^{OP}\left(\totaleclipse\right)$ of $\totaleclipse$\ or its index-two $(3,3,3)$-rotation subgroup.}
\begin{prop}\label{slop}\SlopPropStmt
\end{prop}

We will prove Proposition \ref{slop}, and use it to complete the proof of the main theorem, in Subsection \ref{sloppropproof} below. Subsection \ref{Voronoi sec} lays the foundation for this proof. It begins by reproducing the general definition of the Voronoi cell of a horoball in a horoball packing and culminates in Proposition \ref{no props stop bops}, which characterizes the possible boundary faces of a Voronoi cell of a horoball packing supplied by Lemma \ref{packing} for an arbitrary mixed-platonic orbifold.

\begin{figure}
\includegraphics[width=3.5in]{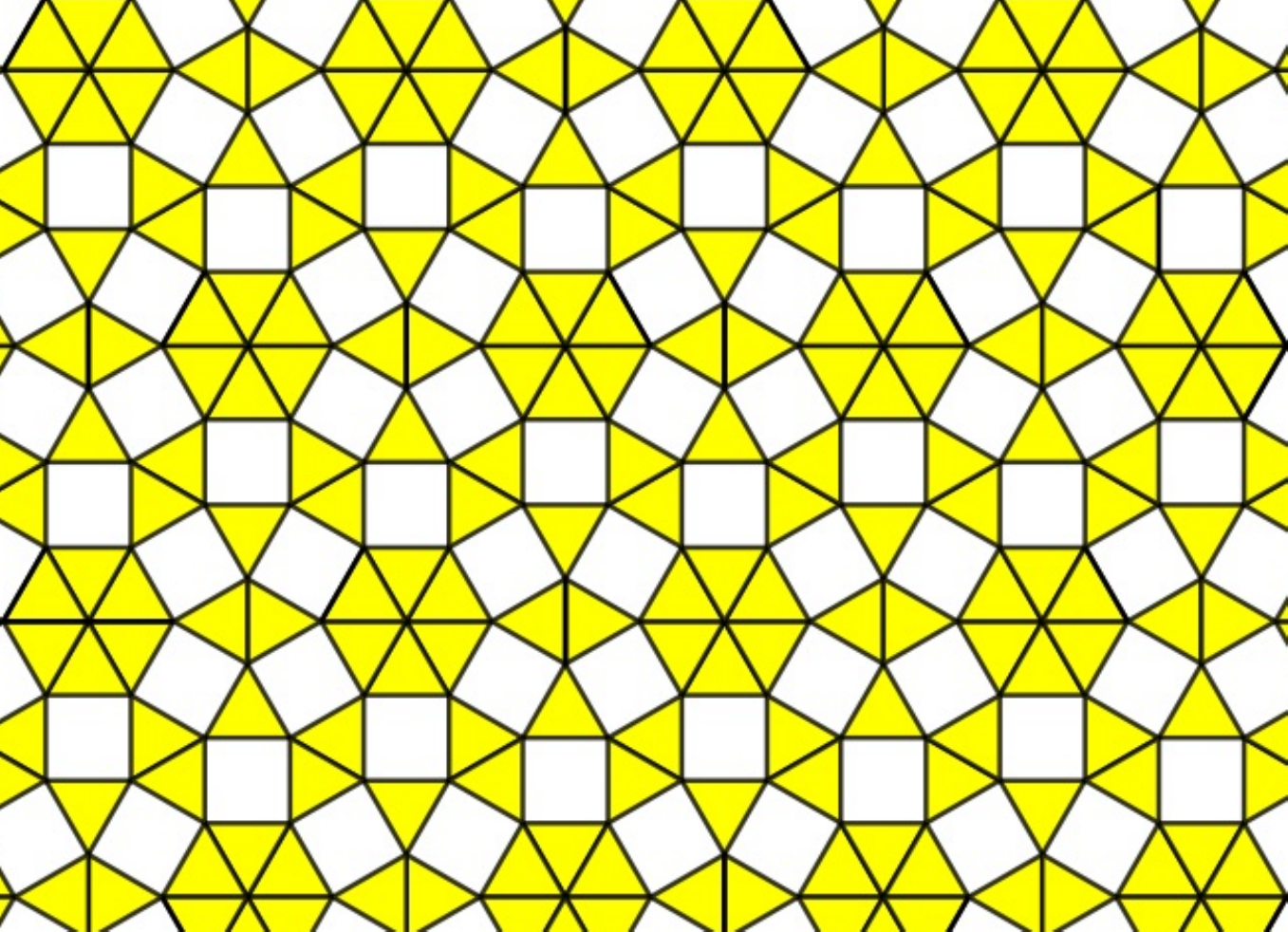}
\caption{\totaleclipse~ - the tiling of the plane where each square is surrounded by four triangles}
\label{fig_total_eclipse}
\end{figure}

\subsection{Voronoi tessellations of horoball packings}\label{Voronoi sec} We begin by reproducing the definition of a standard tool in arguments about cusped hyperbolic $3$-orbifolds since at least the work of Meyerhoff identifying those of minimal volume \cite{Meyerhoff}. Given a horoball $B_p$ centered at an ideal point $p$ of $\mathbb{H}^3$, belonging to a horoball packing $\mathcal{H}$, the \textit{Voronoi cell of $B_p$}, or when necessary, the \textit{Voronoi cell of $B_p$ relative to $\mathcal{H}$}, is defined as:
\begin{align}\label{malign}
	V_p = \{ x\in\mathbb{H}^3\,|\,d(x,B_p) \le d(x,B')\ \mbox{for all}\ B'\in\mathcal{H}\}.
\end{align}
Its \textit{interior} is $\mathit{int}(V_p) = \{x\,|\,d(x,B_p) < d(x,B')\ \forall B'\ne B_p\}$, and its \textit{boundary} is $V_p-\mathit{int}(V_p)$. Some first observations about Voronoi cells, again very standard, are below.

\begin{fact}\label{smackt} Suppose $\mathcal{H}$ is a horoball packing of $\mathbb{H}^3$ and let $\partial_{\infty}\mathcal{H} = \{ p \in \partial \mathbb{H}^3 \, | \, B_p \in \mathcal{H} \}$.\begin{itemize}
	\item For any $p \in \partial_{\infty}\mathcal{H}$, $B_p \subset V_p$ and $\mathbb{H}^3 = \bigcup_{p \in \partial_{\infty}\mathcal{H}} V_p$.
	\item If $\mathcal{H}$ is invariant under the action of a group $\Gamma$ of isometries then so is the collection of Voronoi cells $\{ V_p \, | \, p \in \partial_{\infty}\mathcal{H} \}$.
	\item For $\Gamma$ as above, the stabilizer $\Lambda$ of $p$ in $\Gamma$ equals the stabilizer of $V_p$.
    \item For $\Gamma$, $p$, and $\Lambda$ as above, the restriction to $V_p$ of the quotient map $\mathbb{H}^3\to\mathbb{H}^3/\Gamma$ factors through a map of $V_p/\Lambda$ that is embedding on $\mathit{int}(V_p)/\Lambda$. \end{itemize}
\end{fact}

We next consider the very special case of a symmetric horoball packing centered at the ideal points of a regular ideal tetrahedron or octahedron. We first make a definition.

\begin{definition}\label{effin diction} Given a pair of non-overlapping horoballs $B_p$ and $B_q$, let $\lambda$ be the geodesic joining $p$ to $q$; let $x$ be the point of $\lambda$ equidistant from $B_p$ and $B_q$; and define $H_{pq}$ to be the half-space that is bounded by the geodesic plane that intersects $\lambda$ perpendicularly at $x$ and contains $B_p$. 
\end{definition}

A standard exercise in hyperbolic trigonometry shows that the half-space $H_{pq}$ of Definition \ref{effin diction} consists exactly of the points of $\mathbb{H}^3$ that are at least as close to $B_p$ as to $B_q$. In particular, $H_{pq}$ contains $B_p$. It further follows from this property and the definition that if $B_p$ belongs to a horoball packing $\mathcal{H}$ then:
\begin{align}\label{it's fine} 
    V_p = \bigcap_{p \in \partial_{\infty}\mathcal{H}-\{p\}} H_{pq}
\end{align}
We use this characterization below.

\begin{lemma}\label{crema pronounced "cremma"}
For $P$ a regular ideal tetrahedron or octahedron, let $\mathcal{H}_P$ be the maximal, fully symmetric horoball packing centered at the ideal vertices of $P$ from \Cref{deaf and itchin'}. If $P = T$ is a regular ideal tetrahedron then for a fixed ideal vertex $p$ of $T$, enumerating the other ideal vertices of $T$ as $q_1,q_2,q_3$, the Voronoi cell $V_p = H_{pq_1}\cap H_{pq_2}\cap H_{pq_3}$ of $B_p$ relative to $\mathcal{H}_T$ does not intersect the face of $T$ containing $q_1$, $q_2$, and $q_3$.

If $P = O$ is a regular ideal octahedron then for a fixed ideal vertex $p$ of $O$, the Voronoi cell $V_p$ of $B_p$ is $H_{pq_1}\cap H_{pq_2}\cap H_{pq_3}\cap H_{pq_4}$, where $q_1,q_2,q_3,q_4$ are the other distinct ideal vertices of $O$ that are joined to $p$ by an edge of $O$. $V_p$ intersects only the faces of $O$ that contain $p$.

In each case, $V_p$ has an angle of $2\pi/3$ at each of its edges, which are of the form $H_{pq_i}\cap H_{pq_j}$ for $j\ne i$ such that $q_i$ and $q_j$ are ideal vertices of a common face of $P = T$ or $O$. It has a single vertex $v = \bigcap_i H_{p q_i}$. The nearest-point retraction to $B_p$ takes $v$ to the centroid of the equilateral triangle or square $P\cap \partial B_{p}$, and the edges of $V_p$ into equidistant lines between adjacent vertices of $P\cap\partial B_p$.
\end{lemma}

\begin{proof} In the case of the regular ideal tetrahedron $T$, working in the upper half-space model for $\mathbb{H}^3$ we may assume that the ideal vertices are $p = \infty$, $q_1 = 0$, $q_2 = 1$, and $q_3 = \omega$, where $\omega = \frac{1}{2}(1+i\sqrt{3})$. The horoball packing $\mathcal{H}_T$ then has $B_{\infty}$ as the region above the Euclidean plane of height one, with each $B_{q_i}$ a ball of Euclidean radius $1/2$ tangent to $\mathbb{C}\times\{0\}$ at $q_i$. For each $i$, $H_{pq_i}$ is then the region above a Euclidean hemisphere of radius $1$ centered at $q_i$. 

Points of the Euclidean hemisphere $\partial H_{pq_i}$ are of the form $(z,t)$, where $t = \sqrt{1-\|z-q_i\|^2}$ for $\|z-q_i\|$ the Euclidean distance from $z$ to $q_i$ in $\mathbb{C}$. Such a point $(z,t)$ thus lies in $V_p$ if and only if $z$ is as close to $q_i$ as to any other $q_j$, since if this were not the case it would lie outside of $H_{pq_j}$. The edges of $V_p$ are of the form $H_{pq_i}\cap H_{pq_j}$, for $j\ne i$, and it follows that these consist of points $(z,t)$ such that $z$ is Euclidean-equidistant from both $q_i$ and $q_j$. This also describes their images under projection to $B_{\infty}$ which in the present model is the map $(z,t)\mapsto (z,1)$. A Euclidean geometry computation further shows that the angle between the Euclidean hemispheres $\partial H_{pq_i}$ and $\partial H_{pq_j}$ along their geodesic of intersection is $2\pi/3$.

Finally, the lowest point of $V_p\cap T$ occurs at the point $(z,t)$ such that $z$ has maximum distance from the $q_i$, the triple intersection point: 
\[ v = \bigcap_{i=1}^3 \partial H_{pq_i} = \left\{\left(\frac{3+i\sqrt{3}}{4}, \sqrt{\frac{2}{3}} \right) \right\}. \]
This lies above the geodesic plane containing the face of $T$ opposite $p = \infty$, which is a Euclidean hemisphere of radius $1/\sqrt{3}$. We finally note that the projection of $v$ to $\partial B_{\infty}$ is equidistant from the three vertices $0$, $1$, and $\omega$ of $T\cap\partial B_{\infty}$.

We can likewise embed the regular ideal octahedron $O$ in the upper half-space model for $\mathbb{H}^3$ with its vertex $p$ at $\infty$ and $q_1$ through $q_4$ at $0$, $1$, $1+i$, and $i$ in (say) counterclockwise order. This embedding places the final vertex $p'$ of $O$, the one opposite $p$, at $(1+i)/2$. As for $\mathcal{H}_T$, with this embedding $B_{\infty}\in\mathcal{H}_O$ is the region above the plane at height $1$, the $B_{q_i}$ are Euclidean balls of radius $1/2$, and the $H_{pq_i}$ lie above Euclidean hemispheres of radius $1$. 

Note that in this case, for $i$ and $j$ such that $q_i$ and $q_j$ are opposite vertices of the square bounded by all of the $q_i$ then every point of $\partial H_{pq_i}\cap\partial H_{pq_j}$ lies beneath $\partial H_{pq_k}$ for some $k\ne i,j$, save the unique quadruple intersection point:
\[ v = \bigcap_{i=1}^4 \partial H_{pq_i} = \left\{\left(\frac{1+i}{2}, \frac{1}{\sqrt{2}} \right) \right\}. \]
Therefore $\partial H_{pq_i}\cap\partial H_{pq_j}$ contain an edge of $\bigcap_{i=1}^4 H_{pq_i}$ if and only if $q_i$ and $q_j$ are adjacent vertices of the square, ie.~if and only if they are ideal vertices of a face of $O$. 

The quadruple intersection point $v$ above is also the lowest point of $\bigcap_{i=1}^4 H_{pq_i}$, as its $z$ coordinate maximizes the distance to all $q_i$. The Euclidean hemisphere of radius $1/\sqrt{2}$ centered at $(1+i)/2$, which contains $v$, also contains all $q_i$ in its ideal boundary. Furthermore the reflection in this hemisphere exchanges $p = \infty$ with $p' = (1+i)/2$, fixing the $q_i$. This reflection therefore acts as a symmetry of $O$ which, by the symmetry-invariance of $\mathcal{H}_O$, exchanges $B_p$ with $B_{p'}$. It follows that $H_{pp'}$ is the region above the fixed set of this reflection, so by the above, $H_{pp'}$ contains $\bigcap_{i=1}^4 H_{pq_i}$. Therefore
\[ V_p = H_{pp'} \cap \left(\bigcap_{i=1}^4 H_{pq_i} \right) = \bigcap_{i=1}^4 H_{pq_i}. \]
For the assertion that $V_p$ intersects only the faces of $O$ containing $p$, we note that $\partial H_{pp'}$ separates $V_p$ from the remaining three faces, which all contain $p'$. The assertions about the dihedral angles at edges, and their projections and that of $v$, follow as in the previous case.
\end{proof}

\begin{figure}
\includegraphics[width=4.5 in]{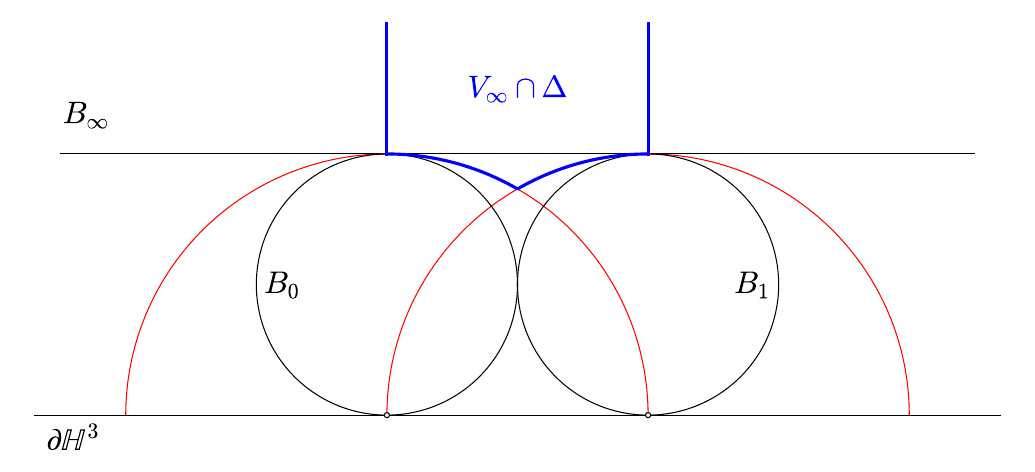}

\caption{$V_\infty \cap \Delta$ outlined in blue, where $\Delta$ is the face with vertices at $0$, $1$, and $\infty$ of a regular ideal tetrahedron or octahedron in standard position as in \Cref{canon pos}. Also showing horoballs centered at $0$, $1$ and $\infty$ belonging to the associated packing $\mathcal{H}_T$ or $\mathcal{H}_O$ from Lemma \ref{crema pronounced "cremma"}, or $\mathcal{H}$ from Lemma \ref{packing}, and cross-sections of equidistant planes in red.}

\label{fig_voronoi}
\end{figure}

For reference, Figure \ref{fig_voronoi} shows the view of $V_{p}\cap \Delta$ ``from the side'', and Figure \ref{single cell Voronoi} below records the ``view from infinity" of $(\partial V_p) \cap T$ and $(\partial V_p) \cap O$ in the two cases of Lemma \ref{crema pronounced "cremma"}, with the polyhedra embedded as in that Lemma's proof with $p$ at $\infty$. Taking ``$P$'' to represent $T$ or $O$ in the respective cases, the portion of $(\partial V_p)\cap P$ contained in $H_{pq_1}$, with $q_1$ at $0$, is shaded in each case, and the bold arcs represent segments of pairwise intersections $\partial H_{pq_i}\cap\partial H_{pq_j}$ that are contained in $(\partial V_p)\cap P$.

\begin{figure}[ht]
\begin{tikzpicture}

\begin{scope}[xshift=-1in, scale=1.4]

\fill [opacity=0.1] (0,0) -- (0.5,0) -- (0.5,0.289) -- (0.25,0.466) -- cycle;

\draw (0,0) -- (1,0) -- (0.5,0.866) -- cycle;
\draw [thick] (0.25,0.433) -- (0.5,0.289) -- (0.75,0.433);
\draw [thick] (0.5,0) -- (0.5,0.289);

\end{scope}

\begin{scope}[xshift=1in, scale=1.4]

\fill [opacity=0.1] (0,0) -- (0.5,0) -- (0.5,0.5) -- (0,0.5) -- cycle;

\draw (0,0) -- (1,0) -- (1,1) -- (0,1) -- cycle;
\draw [thick] (0,0.5) -- (1,0.5);
\draw [thick] (0.5,0) -- (0.5,1);

\end{scope}

\end{tikzpicture}

\caption{The intersection of a Voronoi cell's boundary with a single polyhedron}
\label{single cell Voronoi}
\end{figure}
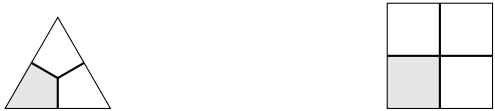

\begin{lemma}\label{tuma, rhymes with "Yuma", as in arizona} Suppose $O = \mathbb{H}^3/\Gamma$ is a mixed-platonic $3$-orbifold, and let $\Pthree$ be a $\Gamma$-invariant tiling of $\mathbb{H}^3$ by regular ideal tetrahedra or octahedra. For any tile $P$ of $\Pthree$, and any $x\in P$, each closest horoball to $x$ among all those in the horoball packing $\mathcal{H}$ supplied by Lemma \ref{packing} is centered at an ideal vertex of $P$. Thus for any ideal vertex $p$ of $P$, with horoball $B_p\in\mathcal{H}$ centered at $p$, $V_p\cap P = V_p^P\cap P$, where $V_p$ and $V_p^P$ are the Voronoi cells of $B_p$ respectively relative to $\mathcal{H}$ and to the packing $\mathcal{H}_P$ from \Cref{deaf and itchin'}. 
\end{lemma}

\begin{proof} First fix a tile $P$ of $\Pthree$ and $x\in P$, and by way of contradiction suppose there is a nearest horoball $B$ of $\mathcal{H}$ to $x$ with its center not at an ideal vertex of $P$. Let $\lambda$ be the arc from $x$ to its nearest point on $B$, and let $P_0$ be the tile containing $\lambda\cap B$. \Cref{packing} implies that $B$ does not intersect $P$, so $P_0\ne P$, and that the center of $B$ is an ideal vertex $p_0$ of $P_0$. The sides of $P_0$ intersecting $B$ contain this ideal vertex and intersect $\partial B$ at right angles. Because $B$ is convex and $\lambda\cap B$ is the closest point on $B$ to $x$, $\lambda$ also intersects $\partial B$ at right angles. Therefore the point $x_0$ where it exits $P_0$ is contained in a face $f_0$ of $P_0$ that does not have the center of $B$ as an ideal vertex. By Lemma \ref{crema pronounced "cremma"}, $x_0$ is closer to a horoball $B_1$ centered at an ideal point of $f_0$ than to $B$. The distance $d(x,B_1)$ from $x$ to $B_1$ is at most the sum of $d(x,x_0)$ and $d(x_0,B_1)$, and this is less than $d(x,B) = d(x,x_0) + d(x_0,B)$. But this contradicts the hypothesis that $B$ is the closest horoball of $\mathcal{H}$ to $x$.

For the final assertion we first recall from the proof of \Cref{packing} that $\mathcal{H} = \bigcup_{P\in\Pthree} \mathcal{H}_P$. The characterization (\ref{it's fine}) therefore immediately implies that $V_p\subseteq V_p^P$, since $V_p$ is an intersection of half-spaces determined by a proper subset of the horoballs in $\mathcal{H}$. To give the other containment for their intersections with $\mathcal{P}$ we appeal to the Lemma's first assertion: for any $x\in P$, since the nearest horoball to $x$ lies in the collection $\mathcal{H}_P$, $x\in V_p$ if and only if $x\in V_p^P$. 
\end{proof}

\begin{corollary}\label{that's amore}
    For a mixed-platonic manifold $M = \mathbb{H}^3$ with a $\Gamma$-invariant packing $\Pthree$ of $\mathbb{H}^3$ by tetrahedra and octahedra, $\Pthree$ is the \mbox{\rm canonical cell decomposition} of the horoball packing $\mathcal{H}$ supplied by \Cref{packing}, in the sense described above Theorem 2.4 of \cite{GoodmanHeardHodgson} (building on the one-cusped case laid out by Epstein--Penner \cite[\S 4]{EpstePen}). 
\end{corollary}

\begin{proof} As described in \cite{GoodmanHeardHodgson}, the canonical cell decomposition of $\mathcal{H}$ is dual to its ``Ford spine'', which per the definition given there is identical to the union of boundaries of the Voronoi cells defined here in (\ref{malign}). Specifically, the canonical $3$-cell dual to a vertex $v$ of the Ford spine is the convex hull of the set of centers of horoballs whose Voronoi cells have $v$ as a vertex. 

The final assertion of Lemma \ref{tuma, rhymes with "Yuma", as in arizona} implies that the set of vertices of Voronoi cells  relative to $\mathcal{H}$ is the union, over all $P\in\Pthree$ and ideal vertices $p$ of $P$, of the vertices of the Voronoi cell $V_p$ of $B_p$ relative to $\mathcal{H}_p$. Lemma \ref{crema pronounced "cremma"} implies that there is a unique such vertex $v_P$ for each $P\in\Pthree$, which lies in the intersection of all $V_p$ for ideal vertices $p$ of $P$. By Lemma \ref{tuma, rhymes with "Yuma", as in arizona}'s first assertion, $v_P$ does not lie in any other Voronoi cell. Therefore $P$---which is the convex hull of its ideal vertices---is the dual to $v_P$.
\end{proof}

Our next result implies that if $\mathcal{T}$ is a peripheral tiling of a mixed-platonic orbifold $O$ then the structure of the Voronoi cell of the corresponding cusp of $O$ is entirely determined by $\mathcal{T}$.

\begin{prop}\label{no props stop bops} Suppose $\mathcal{T}$ is a peripheral tiling of a mixed-platonic orbifold $O = \mathbb{H}^3/\Gamma$. With $\Gamma$ represented in standard position as prescribed in Definition \ref{peripheral}, let $\mathcal{P}$ be a $\Gamma$-invariant tiling of $\mathbb{H}^3$ such that $\mathcal{T} = \mathcal{P}\cap\partial B_{\infty}$, where $B_{\infty}$ is the horoball at $\infty$ of the packing $\mathcal{H}$ supplied by Lemma \ref{packing}. Then the Voronoi cell $V_{\infty}$ of $B_{\infty}$ is a convex polyhedron, in the sense of \cite[\S 3]{Ratcliffe}. Its collection of two-dimensional faces (``sides'', in the terminology of \cite{Ratcliffe})  corresponds bijectively with the set of edges of $\mathcal{P}$ that have one ideal vertex at $\infty$, and has the following further properties:\begin{enumerate}
	\item for any two faces meeting along an edge, the dihedral angle in $V_{\infty}$ is $2\pi/3$ along this edge; and
	\item the collection of faces projects to $\partial B_{\infty}$ along geodesics orthogonal to it, and its image is a tiling of $\mathbb{R}^2$ dual to $\mathcal{T}$: the Voronoi tesselation of the set of vertices of $\mathcal{T}$, defined as in (\ref{malign}) but with vertices replacing horoballs.\end{enumerate}
Each two-cell of the tiling of $\partial V_{\infty}$ is drawn from among the tile types pictured in Figure \ref{Voronoise}, labeled according to the sequence of types of tiles of $\mathcal{P}$ encountered in cyclic order around the corresponding edge of $\mathcal{P}$.
\end{prop}

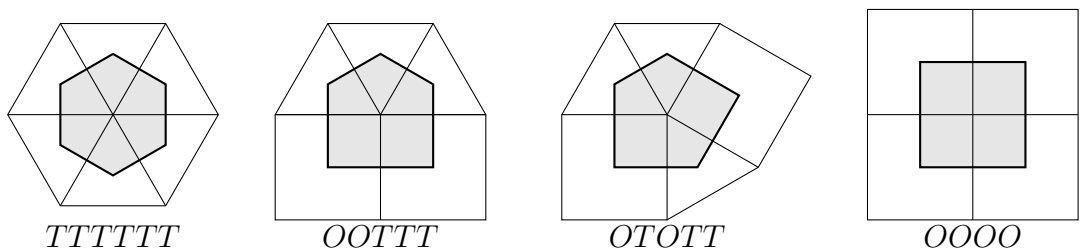
\begin{figure}[ht]
\begin{tikzpicture}

\begin{scope}[xshift=-2.2in, scale=0.7]

\draw (2,0) -- (1,1.732) -- (-1,1.732) -- (-2,0) -- (-1,-1.732) -- (1,-1.732) -- cycle;
\draw (-2,0) -- (2,0);
\draw (-1,1.732) -- (1,-1.732);
\draw (-1,-1.732) -- (1,1.732);

\fill [opacity=0.1] (1,0.577) -- (0,1.1547) -- (-1,0.577) -- (-1,-0.577) -- (0,-1.1547) -- (1,-0.577) -- cycle;
\draw [thick] (1,0.577) -- (0,1.1547) -- (-1,0.577) -- (-1,-0.577) -- (0,-1.1547) -- (1,-0.577) -- cycle;

\node [below] at (0,-1.9) {$TTTTTT$};

\end{scope}

\begin{scope}[xshift=-0.8in, scale=0.7]

\draw (2,0) -- (1,1.732) -- (-1,1.732) -- (-2,0) -- cycle;
\draw (-2,0) -- (2,0);
\draw (0,0) -- (-1,1.732);
\draw (0,0) -- (1,1.732);
\draw (-2,0) -- (-2,-2) -- (2,-2) -- (2,0);
\draw (0,0) -- (0,-2);

\fill [opacity=0.1] (1,-1) -- (1,0.577) -- (0,1.1547) -- (-1,0.577) -- (-1,-1) -- cycle;
\draw [thick] (1,-1) -- (1,0.577) -- (0,1.1547) -- (-1,0.577) -- (-1,-1) -- cycle;

\node [below] at (0,-1.9) {$OOTTT$};

\end{scope}

\begin{scope}[xshift=2.3in, scale=0.7]

\draw (2,2) -- (-2,2) -- (-2,-2) -- (2,-2) -- cycle;
\draw (-2,0) -- (2,0);
\draw (0,-2) -- (0,2);

\fill [opacity=0.1] (1,1) -- (-1,1) -- (-1,-1) -- (1,-1);
\draw [thick] (1,1) -- (-1,1) -- (-1,-1) -- (1,-1) --cycle;

\node [below] at (0,-1.9) {$OOOO$};

\end{scope}

\begin{scope}[xshift=0.7in, scale=0.7]

\draw (1,1.732) -- (-1,1.732) -- (-2,0) -- (-2,-2) -- (0,-2) -- (0,0) -- cycle;
\draw (-2,0) -- (0,0) -- (-1,1.732);
\draw (0,0) -- (1.732,-1) -- (2.732,0.732) -- (1,1.732);
\draw (0,-2) -- (1.732,-1);

\fill [opacity=0.1] (0.577,-1) -- (1.366,0.366) -- (0,1.1547) -- (-1,0.577) -- (-1,-1) -- cycle;
\draw [thick] (0.577,-1) -- (1.366,0.366) -- (0,1.1547) -- (-1,0.577) -- (-1,-1) -- cycle;

\node [below] at (0,-1.9) {$OTOTT$};

\end{scope}

\end{tikzpicture}

\caption{The four types of possible projected Voronoi faces of a peripheral tiling.}
\label{Voronoise}
\end{figure}

\begin{proof} For $\mathcal{P}$ and $\mathcal{H}$ as in the Proposition statement, and any $P\in\mathcal{P}$, let $\mathcal{H}_P$ be the finite set consisting only of those horoballs of $\mathcal{H}$ which are centered at ideal vertices of $P$. By Lemma \ref{tuma, rhymes with "Yuma", as in arizona}, if $P\in\mathcal{P}$ has $\infty$ as an ideal vertex then $V_{\infty}\cap P = V_p\cap P$, where $V_p$ is the Voronoi cell of $B_{\infty}$ relative to the horoball packing $\mathcal{H}_P$ of \Cref{deaf and itchin'}. This implies that $V_{\infty}$ is a polyhedron: for any compact set $K$ intersecting $V_{\infty}$, $K\cap V_{\infty}$ is contained in the union $P_1\cup\hdots\cup P_n$ of finitely many members of $\mathcal{P}$, and for any such $P_j = P$, by the claim there are either three or four geodesic planes $H_{pq_i}$ (notation from the proof of Lemma \ref{crema pronounced "cremma"}, with $p = \infty$) whose union contains $\partial V_{\infty}\cap P$. It therefore follows that $K\cap\partial V_{\infty}$ intersects only the finite collection of ``sides'' (per \cite{Ratcliffe}) that have the form $H_{pq_i}\cap V_{\infty}$, taken over all such $P = P_j$ and their associated $H_{pq_i}$.

For the characterization of the sides, ie.~two-dimensional faces, we begin by recalling that by the definition of $\mathcal{H}$ in Lemma \ref{packing}, for each polyhedron $P$ as above and each edge $e$ of $P$ with an ideal endpoint at $\infty$, there is another horoball $B_q$ of $\mathcal{H}_P$ centered at the other ideal endpoint $q$ of $e$ and a point of tangency $B_{\infty}\cap B_q = \{x\}\subset e$. Because $x$ is a point of tangency, there is an open neighborhood of $x$ in the totally geodesic plane $\partial H_{\infty q}$ (as in Definition \ref{effin diction}) consisting of points for which $B_q$ and $B_{\infty}$ are the only closest members of $\mathcal{H}$; which therefore lie in $F_e:= H_{\infty q}\cap\partial V_{\infty}$. The intersection of $F_e$ with each polyhedron $P'\in\mathcal{P}$ containing $e$ has the form of one of the kites pictured in Figure \ref{single cell Voronoi}, and the union of these kites forms one of the polygons of Figure \ref{Voronoise}. (These reflect all combinatorial possibilities for the arrangement of regular ideal tetrahedra and octahedra containing a single edge.)

The assertion about the dihedral angle of $V_{\infty}$ along any of its edges, and the characterization of $\partial V_{\infty}$ as the Voronoi tessellation of the vertices of $\mathcal{T}$, both now follow directly from Lemma \ref{crema pronounced "cremma"}.
\end{proof}

The observation below concerns the pentagons of Figure \ref{Voronoise}.

\begin{lemma}\label{restrictionz} Let $B$ be a horoball of a packing $\mathcal{H}$ associated to an orientable mixed-platonic orbifold $\mathbb{H}^3/\Gamma$ as in Lemma \ref{packing}, centered at ideal vertices of a $\Gamma$-invariant tiling $\mathcal{P}$, and let $V$ be the Voronoi cell of $B$ relative to $\mathcal{H}$. For any pentagonal faces $F, F'$ of the tiling of $\partial V$ described in Proposition \ref{no props stop bops}, such that $\gamma(F) = F'$ for some $\gamma\in\Gamma$, $F$ and $F'$ have the same type, and $\gamma$ takes the unique edge of $F$ that joins the centers of two octahedra (if $F$ is of type $OOTTT$) or tetrahedra (if type $OTOTT$) to the corresponding edge of $F'$. In particular, the only non-trivial self-isometry of either pentagon type is a reflection fixing the edge specified above.
\end{lemma}

\begin{proof} This follows essentially by inspection of Figure \ref{Voronoise}, upon recalling that the lighter-outlined squares and triangles there represent the horospherical cross-sections of tiles of $\mathcal{P}$ by $\partial B$. Every element of $\Gamma$ takes tiles of $\mathcal{P}$ to tiles of $\mathcal{P}$, and if $\gamma$ takes $F$ to $F'$ then it takes the edge of $\mathcal{P}$ corresponding to $F$, with one ideal point at the center of $B$, to the edge of $\mathcal{P}$ corresponding to $F'$.\end{proof}

\subsection{The proof of \Cref{slop}}\label{sloppropproof} We begin with a Lemma that builds on the previous subsection's results.

\begin{lemma}\label{mo restrictionz} Let $B$ be a horoball of a packing $\mathcal{H}$ associated to an orientable mixed-platonic orbifold $\mathbb{H}^3/\Gamma$ as in Lemma \ref{packing}, centered at ideal vertices of a $\Gamma$-invariant tiling $\mathcal{P}$, and let $V$ be the Voronoi cell of $B$ relative to $\mathcal{H}$. For a face $F$ of $V$ and an edge $e$ of $F$ such that an element $\gamma\in\Gamma$ outside the stabilizer $\Lambda$ of $B$ stabilizes $e$ and $F$, let $F_0$ be the face of $V$ that shares the edge $e$, and let $B_0\ne B$ be the horoball that is tangent to $B$ in the interior of $F_0$. Then $\gamma$ stabilizes $B_0$ and its ideal point.\end{lemma}

\begin{proof} For $B$, $V$, $F$, $e$, and $\gamma$ as in the hypothesis, since $\gamma$ lies outside $\Lambda$ it is non-trivial. Because it stabilizes $F$ and $e$ it must act non-trivially on $e$: if it fixed $e$ then since it preserves the normal direction pointing into $F$ it would also fix $F$, hence since it preserves orientation it would fix all of $\mathbb{H}^3$. Thus it acts on $e$ as reflection through its midpoint. Since it stabilizes $F$ but not $B$ it must exchange $B$ with the horoball $B'$ of $\mathcal{H}$ tangent to $B$ in the interior of $F$, hence acting as a $180$-degree rotation around the perpendicular bisector of $e$ in the hyperplane containing $F$.

By Proposition \ref{no props stop bops}, each Voronoi cell containing $e$ has a dihedral angle of $2\pi/3$ along $e$; therefore there are exactly three such Voronoi cells. Since each of $V$, the Voronoi cell $V'$ of $B'$, and the Voronoi cell $V_0$ of the horoball $B_0$ identified in the Lemma's statement contains $e$, these must be the only three. Since $\gamma$ takes $e$ to itself it preserves the collection of horoballs whose Voronoi cells contain $e$. It exchanges $B$ with $B'$, by the paragraph above, so it must stabilize $B_0$ and therefore also its ideal point.
\end{proof}

The result below takes for granted the fact that a discrete group $\Gamma$ such that $\mathbb{H}^3/\Gamma$ has \emph{one cusp} acts transitively on an associated $\Gamma$-invariant horoball packing.

\begin{corollary}\label{kablooie!}
    Suppose $O = \mathbb{H}^3/\Gamma$ is a one-cusped mixed-platonic orbifold and $B$ is a horoball  of the packing $\mathcal{H}$ associated to $O$ by Lemma \ref{packing}, with Voronoi cell $V$ and stabilizer $\Lambda\in\Gamma$. If there exist a face $F$ of $V$ and edge $e$ of $F$ both stabilized by some $\gamma\in\Gamma - \Lambda$, then $\Lambda$ contains two-torsion. 
\end{corollary}

\begin{proof}
With the given hypothesis, Lemma \ref{mo restrictionz} identifies a horoball $B_0$ of $\mathcal{H}$ such that $\gamma$ stabilizes $B_0$ and its ideal point. Because $O$ has one cusp, $\Gamma$ acts transitively on $\mathcal{H}$, so there exists $\delta\in\Gamma$ taking $B_0$ to $B$. Then $\delta\gamma\delta^{-1}$ is an order-two element of $\Lambda$.
\end{proof}

Recall the statement of \Cref{slop}:

\theoremstyle{plain}
\newtheorem*{slop prop}{\Cref{slop}}
\begin{slop prop}\SlopPropStmt\end{slop prop}

\begin{proof}
Suppose that there is an orientable, one-cusped orbifold $O = \mathbb{H}^3/\Gamma$ such that $\totaleclipse$\ is the  peripheral tiling of $O$ in the sense of Definition \ref{peripheral}. That is, with $\Gamma$ in standard position as in Definition \ref{canon pos}, $\totaleclipse = \mathcal{P}\cap \partial B_{\infty}$, where $\mathcal{P}$ is a $\Gamma$-invariant tiling by regular ideal tetrahedra and octahedra and $B_\infty$ is the horoball centered at $\infty$ of the $\Gamma$-invariant horoball packing $\mathcal{H}$ supplied by \Cref{packing}. Let $V_\infty$ be the Voronoi cell of $B_{\infty}$.  If $\Lambda < \Gamma$ is the stabilizer of $\infty$ then restricting its action to $\partial B_{\infty}$ exhibits $\Lambda$ as a subgroup of $\text{Sym}^{OP}\left(\totaleclipse\right)$.

The Voronoi cells of the horoballs of $\mathcal{H}$ tile $\mathbb{H}^3$ and, since $O$ has a single cusp, $\Gamma$ acts transitively on these cells.  
This implies that the restriction of the universal cover to $V_\infty$ factors through a map of $V_\infty/\Lambda$ ultimately mapping onto $O$.  The interior of $V_\infty$ embeds in $O$ under this map.   These assertions follow from Fact \ref{smackt}.  The boundary of $V_\infty$ is tiled by totally geodesic hexagons and pentagons whose edges have finite length.  $\Lambda$ acts on $\partial V_\infty$ preserving the tiling, taking edges to edges and faces to faces.  If $\gamma \in \Lambda$ and $e$ is an edge of this tiling, there are only two possibilities for the  map $\gamma : \gamma^{-1}(e) \to e$ and they differ by the order-2 reflection of $e$ through its midpoint.  This means that each edge can be broken into two edges by adding midpoints as vertices and $\Lambda$ acts on this graph of half-edges taking edges to edges and vertices to vertices.  The dihedral angle of $\partial V_\infty$ at any given half-edge is $2\pi/3$ and it follows that the total angle in $O$ around the image of a given half-edge $e$ of $\partial V_\infty$ is $m\cdot\frac{2\pi}{3}$, where $m$ is the number of distinct $\Lambda$-orbits of half-edges identified to $e$ in $O$.  Because $O$ is a hyperbolic orbifold, the total angle around any half-edge quotient in $O$ must be an integer submultiple of $2\pi/n$ of $2\pi$, where $n$ is the order of the edge stabilizer in $\Gamma$.  This means that, for any half-edge $e$ of $\partial V_\infty$, the number of distinct $\Lambda$-orbits of half-edges identifed to $e$ in $O$ is either one or three.

Notice that every (full) edge of $\partial V$, for any given Voronoi cell associated to $\mathcal{H}$, either has endpoints contained in a pair of adjacent tetrahedra or in an adjacent pair of an octahedron and a tetrahedron.  Edges of the first type will be called {\em short} and of the second type {\em long}.   Because $\Gamma$ preserves the mixed platonic tiling of $\mathbb{H}^3$, the action of $\Gamma$ on the set of edges of faces of Voronoi cells respects this edge classification.  In particular, any orbit containing a short half-edge contains exclusively short half-edges.

\begin{figure}
\includegraphics[width=4.5 in]{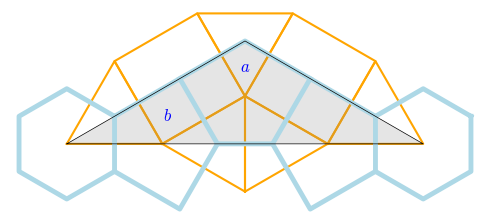}
\caption{A fundamental domain for the $(2,3,6)$-rotation group $\text{Sym}^{OP}\left(\totaleclipse\right)$ of symmetries of $\totaleclipse$. }
\label{fig_domain}
\end{figure}

Suppose that $\Lambda$ is the full OP symmetry group of $\totaleclipse$.  Then $\Lambda$ is a $(2,3,6)$-rotation group and the gray triangle $D$ in Figure \ref{fig_domain} is a fundamental domain for its action on $\partial V_\infty$.   In the figure, the tiles of $\totaleclipse$\ which intersect $D$ are shown in orange and the faces of $\partial V_\infty$ which intersect $D$ are shown in light blue.  Let $\sigma$ be the order-2 elliptic isometry whose axis is vertical and pierces the horizontal edge of $D$.

Inspecting the Figure \ref{fig_domain}, we find that the tiling of $\partial V_\infty$ by pentagons and hexagons has two $\Lambda$-orbits of pentagons, represented by the pentagons labeled $a$ and $b$.  There is a single $\Lambda$-orbit of hexagons.  Name the left most hexagon in Figure \ref{fig_domain} as $h$ and set $t=h \cap D$.  Let $V_h$, $V_a$, and $V_b$ be the Voronoi cells that meet $V_\infty$ along $h$, $a$, and $b$.  Let $e$ be the upper half of the short edge of $b$.  Our goal is to show, for a contradiction, that the number of distinct $\Lambda$-orbits of half-edges identifed to $e$ in $O$ cannot be one or three.
  
 Since $\Gamma$ acts transitively on the tiling by Voronoi cells, there is an element $\gamma_1 \in \Gamma$ which takes $V_\infty$ to $V_b$.  Let $X$ be the pentagon $\gamma_1^{-1}(b)$.  Then $\gamma_1(X)=b$.  Since $X$ belongs to the $\Lambda$-orbit of either $a$ or $b$, by precomposing $\gamma_1$ by an element of $\Lambda$, we may assume $X \in \{ a, b \}$.  Similarly, there is an element $\gamma_2 \in \Gamma$ which takes $V_\infty$ to $V_h$.   As before, we may assume that $\gamma_2^{-1}(t) \subset D$ and, since $\gamma_2$ acts as an OR isometry (using the induced boundary orientation from $V_\infty$), it must be true that $\gamma_2^{-1}(t)$ is the triangle $\sigma(h) \cap D$.
 
 The dihedral angles of $V_\infty$, $V_b$, and $V_h$ at $e$ are all $2\pi/3$, so $V_b$ and $V_h$ intersect at a face $F$ that contains $e$.  Because $\gamma_2^{-1}(V_h)=V_\infty$ and $\gamma_2^{-1}(h)=\sigma(h)$, the pentagon $\sigma(b)$ is $\gamma_2^{-1}(F)$.  Therefore, $V_\infty=\gamma_1^{-1}(V_b)$ has a pentagonal face across $\gamma_1^{-1}(e)$ from $X$ and so $X$ must be $a$.
 
Note that there are a total of 8 $\Lambda$-orbits of long half-edges.  However, upon combining the effects of $\gamma_1$ on the edges of $a$ and $b$ with the action of $\Lambda$, we find that four of these orbits are identified under the action of $\Gamma$.  Since there should only be three, this is a contradiction.

\begin{figure}
\includegraphics[width=4.5 in]{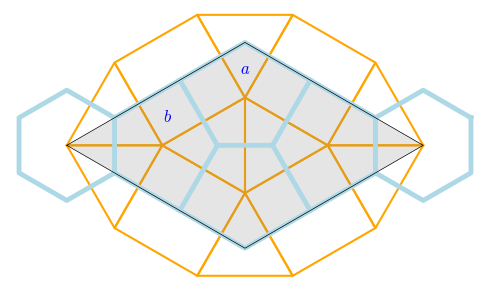}
\caption{A fundamental domain for the $(3,3,3)$-rotation group that is the unique index-two subgroup of $\text{Sym}^{OP}\left(\totaleclipse\right)$.}
\label{fig_domain2}
\end{figure}

To finish, suppose that $\Lambda$ is the index-2 $(3,3,3)$-rotation subgroup of $\text{Sym}^{OP}\left(\totaleclipse\right)$.  The gray diamond $D'$ in Figure \ref{fig_domain2} is a fundamental domain for the action of $\Lambda$ on $\partial V_\infty$.  The tiling of $\partial V_\infty$ by geodesic hexagons and pentagons has four $\Lambda$-orbits of pentagons which are represented by $a$, $b$, and their images under $\sigma$.  We name the leftmost hexagon in Figure \ref{fig_domain2} as $h$, as in the $(2,3,6)$ case, and let $t'$ be the triangle $h\cap D'$.

Naming the Voronoi cells adjacent to $V_{\infty}$ as $V_a$, $V_b$, and $V_h$ as before, and taking $\gamma_2\in\Gamma$ to be an element with $\gamma_2 (V_{\infty}) = V_h$ and $\gamma_2^{-1}(t')\in D'$, in this case (unlike the previous one) it is \textit{a priori} possible that $\gamma_2(t') = t'$. However if this were so then $\gamma$ would stabilize $h$ and the edge $e$ of $h$ contained in $t'$, so $\Lambda$ would contain two-torsion by Corollary \ref{kablooie!}, contradicting that it is a $(3,3,3)$-rotation group. Therefore $\gamma_2^{-1}(t') = \sigma(t')$, analogous to the first case.

Now following the argument from the first case, we find an element $\gamma_1 \in \Gamma$ which takes $V_\infty$ to $V_b$ so that $\gamma_1(X)=b$ for some $X \in \{ a, \sigma(a) \}$.  Applying the same argument to $\sigma(b)$, on the right half of the domain $D'$, we find an element $\gamma_1' \in \Gamma$ which takes $V_\infty$ to $V_{\sigma(b)}$ so that $\gamma_1'(Y)=\sigma(b)$ for some $Y \in \{ a, \sigma(a) \}$.  Now, if $X=Y$, then $\gamma_1'\gamma_1^{-1}(b)=\sigma(b)$ and $\gamma_1'\gamma_1^{-1}(V_b)=V_{\sigma(b)}$ and together these imply that $\gamma_1'\gamma_1^{-1}$ fixes $\infty$.  But then $\gamma_1'\gamma_1^{-1} \in \Lambda$ and identifies the distinct $\Lambda$-orbits represented by $b$ and $\sigma(b)$, a contradiction.  Therefore, $X$ and $Y$ must be distinct in $\{ a, \sigma(a) \}$.

From here, our proof assesses the number of $\Lambda$-orbits of long half-edges which are identified to single half-edges in $O$.  The argument is undisturbed by the horizontal reflection of the diamond $D'$, so we may assume $X=a$ and $Y=\sigma(a)$.

Now again, upon combining the effects of $\gamma_1$ and $\gamma_1'$ on the $\Lambda$-orbits of long half-edges with the action of $\Lambda$, we find that more than three of these orbits are identified to a single edge in $O$ under the action of $\Gamma$.  Since there should only be three, this is a contradiction and our proof is complete.
\end{proof}

\begin{corollary}\label{snore oh larry} There is no mixed-platonic knot whose complement has hidden symmetries and peripheral tiling $\totaleclipse$.
\end{corollary}

\begin{proof} Suppose $M = \mathbb{H}^3/\Gamma_K$ is the complement of a mixed-platonic knot $K$, with peripheral tiling $\totaleclipse$, and let $B$ be a horoball of the associated packing from Definition \ref{peripheral}. Recalling that each equilateral triangle and square edge of $\totaleclipse$\ has length $1$ on $\partial B$, we compute that the shortest distance between distinct vertices of $\totaleclipse$\ that are shared by six triangles is $2+\sqrt{3}$. Since the peripheral subgroup $\Lambda_K$ of $M$ preserves $\totaleclipse$, it takes any such vertex to another. By Proposition \ref{neilolarry}, an element $\mu$ of $\Lambda_K$ representing a meridian of $K$ has translation length exactly $2+\sqrt{3}$ and hence takes any such vertex $v$ to another vertex at minimal distance from it.

If $M$ has hidden symmetries then it covers a rigid cusped orbifold $O = \mathbb{H}^3/\Gamma$, for a discrete group $\Gamma$ containing $\Gamma_K$, with a $(2,3,6)$- or $(3,3,3)$-rigid cusp \cite{NeilOrbiCusps}. Proposition \ref{commensurator} implies that $O$ also has $\totaleclipse$\ as its peripheral tiling. Let $\Lambda$ be the peripheral subgroup of $\Gamma$ containing $\Lambda_K$. If $\Lambda$ is a $(2,3,6)$-rotation group then we claim it must be the full symmetry group $\text{Sym}^{OP}\left(\totaleclipse\right)$ of $\totaleclipse$, contradicting Proposition \ref{no props stop bops}.

To see this, let $\rho_0$ be an order-six element of $\Lambda$, and let $\rho_1 = \mu\rho_0\mu^{-1}$. The fixed points of the $\rho_i$ are each valence-six vertices of $\totaleclipse$, at distance $2+\sqrt{3}$. Hence inspecting $\totaleclipse$\ we find that $\sigma = \rho_0\rho_1$ is an order-three rotation fixing the center of one of the equilateral triangles adjacent to three squares that is closest to the midpoint of the segment joining the fixed point of $\rho_0$ to that of $\rho_1$. Then $\langle\rho_0,\sigma\rangle  = \text{Sym}^{OP}\left(\totaleclipse\right)$, by co-area for example, proving the claim.

We now claim that if $\Lambda$ is a $(3,3,3)$-rotation group then it is the order-two $(3,3,3)$-subgroup of $\text{Sym}^{OP}\left(\totaleclipse\right)$, again contradicting Proposition \ref{no props stop bops}. To see this, let $\rho$ be an order-three rotation of $\Lambda$ with fixed point $c$. The axis of translation of $\mu$ is parallel to one of the three sides of the equilateral triangle formed by the three valence-six vertices of $\totaleclipse$\ closest to $c$. Then either $\mu\rho$ or $\mu\rho^{-1}$ fixes an endpoint of this side, a valence-six vertex $v$, and hence is an order-three rotation around $v$. Calling this element $\sigma$, we now find that $\langle \rho,\sigma\rangle$ is the index-two subgroup of $\text{Sym}^{OP}\left(\totaleclipse\right)$, again by co-area for example. This proves the result.
\end{proof}

The result above is the final ingredient in the proof of this paper's main theorem: 

\begin{theorem}\label{main_thm}
\NoMPKHidSym 
\end{theorem}

\begin{proof} Suppose $M = \mathbb{H}^3/\Gamma$ is a mixed-platonic knot complement with hidden symmetries. Since it is mixed-platonic it is non-arithmetic, by \Cref{non arithmetic}, so its commensurator $\mathrm{Comm}(\Gamma)$ is discrete. By \Cref{packing}, the $\Gamma$-invariant tiling $\mathcal{P}$ by regular ideal tetrahedra and octahedra is also $\mathrm{Comm}(\Gamma)$-invariant. Let $\mathrm{Comm}^+(\Gamma)$ be the orientation-preserving subgroup of $\mathrm{Comm}(\Gamma)$. The orientable commensurator quotient $O := \mathbb{H}^3/\mathrm{Comm}^+(\Gamma)$ is then itself mixed-platonic. 

Since $M$ has hidden symmetries, $O$ has a rigid cusp by \cite[Prop.~9.1]{NeumannReidArith}. It then follows from Theorem 1.1 of \cite{NeilOrbiCusps} and the classification of Euclidean orbifolds that $O$ has either a $(2,3,6)$- or $(3,3,3)$-cusp. With $O$ represented in standard position as in \Cref{canon pos}, let $\mathcal{T}$ be the peripheral tiling from \Cref{peripheral}, and let $\Lambda$ be the peripheral subgroup of $\mathrm{Comm}^+(\Gamma)$ stabilizing the horoball $B_{\infty}$ centered at $\infty$ of the horoball packing from \Cref{packing}. Then $\Lambda$ has order-three rotations and leaves $\mathcal{T}$ invariant. 

Let $\Lambda_0<\Lambda$ be the peripheral subgroup of $\Gamma$ stabilizing $B_{\infty}$. By \Cref{neilolarry}, the element of $\Lambda_0$ representing a meridian acts on $\partial B_{\infty}$ as a translation by at most $2+\sqrt{3}$; hence this is an upper bound for the minimum translation length of infinite-order elements of $\Lambda$. \Cref{And then there were two} therefore implies that $\mathcal{T}$ is one of the three tilings generated from those in Figure \ref{dumb bell} as in Remark \ref{embark}. However none of these can be the peripheral tiling of $M$: by \Cref{pom}, \Cref{no ace}, and \Cref{snore oh larry}, respectively, moving from left to right across the Figure. This contradiction finishes the proof.
\end{proof}

\section{Questions}\label{sec:questions}

The full class of mixed-platonic manifolds remains generally mysterious beyond the known examples described in \Cref{tinderbox}. For instance the following is open to our knowledge:

\begin{question} Are there mixed-platonic knots other than $12n706$?  Are there infinitely many?  
\end{question}

By \Cref{non arithmetic}, every mixed-platonic knot complement has the quartic invariant trace field $\mathbb{Q}(i,\sqrt{-3})$. While the hybridization construction discussed in \Cref{ex:mph} produces infinite families of link complements with this invariant trace field, and one can also produce infinite families of \emph{two component} link complements having the same quartic invariant trace fields (see \cite{EricJasonAlgebraicInvariants}), it seems that there are no known infinite families of knot complements with the same trace field. Indeed, while small finite families of knot complements with the same quartic trace fields have been observed computationally (for example the complements of $6a3$, $7a1$, and $9a7$ in the HTLinkExteriors census), these computations could best be described as observations and do not seem to lead to any kind of constructive method.

In another algebraic direction, given that by \Cref{neilolarry}, mixed-platonic knot complements have representation in the ring of integers of the cyclotomic field $\mathbb{Z}[\zeta_{12}]$, we ask:

\begin{question} Are there knot complements other than those of the figure-eight and $12n706$ with representations in the ring of integers of a cyclotomic field? 
\end{question} 

One way to shed light on the questions above would be to answer the one below, which can likely be attacked using current computational methods.

\begin{question} What is the classification of one-cusped mixed-platonic manifolds of low volume/complexity? What is their \mbox{\rm commensurability} classification?
\end{question}

Regarding the commensurability classification, we note that any two platonic manifolds that decompose into copies of the same platonic solid---eg.~any two tetrahedral manifolds---are commensurable, but this is very much not the case for mixed-platonic manifolds. For instance, the Boyd knot complement of \Cref{ex: Boyd knot} is not commensurable with $s913$ from \Cref{s913}: the former has a mixed-platonic commensurator quotient, by \Cref{commensurator}, whereas the latter does not (cf.~\Cref{v2774}).

Further motivated by \Cref{s913} and \Cref{commensurator}, we ask:

\begin{question}
Does every mixed-platonic manifold $M = \mathbb{H}^3/\Gamma$ have a \mbox{\rm unique} decomposition into regular ideal tetrahedra and octahedra, up to self-isometry?
\end{question}

The results of Sections \ref{perti} and \ref{eclipse} raise the following question:

\begin{question}
For $\Lambda$ equal to each of the $(2,3,6)$-, $(2,4,4)$-, and $(3,3,3)$-rotation groups, does there exist a one-cusped mixed-platonic orbifold with peripheral subgroup $\Lambda$? If so, what is the smallest-volume such example?
\end{question}

Since the figure-8 knot complement and the dodecahedral knot complements are the only knot complements known to admit hidden symmetries and the only knot complements known to contain immersed closed totally geodesic surfaces, we can also ask: 

\begin{question} \label{ques:hiddentg} Does every knot complement which admits hidden symmetries contain a closed, possibly immersed, totally geodesic surface? 
\end{question} 


\end{document}